\documentclass{amsart}
\usepackage{amssymb}
\usepackage{amsmath,amssymb, amsthm, fourier,mathrsfs}
\usepackage{enumitem,dsfont}
\usepackage{graphicx}
\usepackage{xcolor}
\usepackage{amsaddr}
\usepackage{slashed}
\usepackage{mathtools}
\mathtoolsset{showonlyrefs}

\newtheorem{Theorem}{Theorem}[section]
\newtheorem{Definition}[Theorem]{Definition}
\newtheorem{Corollary}[Theorem]{Corollary}
\newtheorem{Lemma}[Theorem]{Lemma}
\newtheorem{Proposition}[Theorem]{Proposition}
\newtheorem{Remark}{Remark}[section]

\title{Fractional Sobolev Spaces for the Singular-perturbed Laplace Operator in the $L^p$ setting}

 \numberwithin{equation}{section}

\author[V. Georgiev]{Vladimir Georgiev}

\address[V.Georgiev]{
Department of Mathematics,
University of Pisa,
Largo Bruno Pontecorvo 5,
I - 56127 Pisa, Italy}
\address{
Faculty of Science and Engineering, Waseda University,
3-4-1, Okubo, Shinjuku-ku, Tokyo 169-8555, Japan}
\address{
Institute of Mathematics and Informatics,  Bulgarian Academy of Sciences, Acad. G. Bonchev Str., Block 8, Sofia, 1113, Bulgaria
}
\email{georgiev@dm.unipi.it}
\author[M.Rastrelli]{Mario Rastrelli} 
\address[M.Rastrelli]{
Department of Mathematics,
University of Pisa,
Largo Bruno Pontecorvo 5,
I - 56127 Pisa, Italy}
\email{mario.rastrelli@phd.unipi.it}
\address{Department of Applied Physics, Waseda University
Tokyo 169-8555, Japan}

\thanks{
 V.G. and M.R. were partially supported by  GNAMPA 2024.
V.G. was partially supported by    the Top Global University Project, Waseda University, by the University of Pisa, Project PRA 2022 85 and by Institute of Mathematics and Informatics, Bulgarian Academy of Sciences. 
{ The authors would like to thank Raffaele Scandone who provided details comments and  useful ideas to improve the proof of Lemma \ref{l.dsgr06}. }
}


\begin{document}

\begin{abstract}
We study the perturbed Sobolev spaces ${H^{s,p}_\alpha(\mathbb{R}^d)}$, associated with singular perturbation $\Delta_\alpha$ of Laplace operator in Euclidean space of dimensions 2 and 3. We extend the $L^2$ theory of perturbed Sobolev space to  the $L^p$ case, finding an analogue description in terms of standard Sobolev spaces. This enables us to extend the Strichartz estimates to the energy space and to treat the {local well-posedness} of the {Nonlinear Schr\"odinger equation} associated with this singular perturbation, with the contraction method.
\end{abstract}

\maketitle

\section{Introduction }
 In this paper, we study the fractional domains of the singular-perturbed Laplacian in dimensions 2 and 3 and we investigate the local existence of the solution of the Nonlinear Schr\"odinger equation with power-type nonlinearity in these spaces. 
With singular-perturbed Laplacian, we refer to the operator $-\Delta_{\alpha}$, with $\alpha\in (-\infty,+\infty]$, that is a delta-like perturbation of the Laplacian in $\mathbb{R}^d$, for $d=2,3$.

 To be more specific, we treat the operator $-\Delta+q\delta_0$, where $-\Delta$ is the standard Laplace operator, $q$ is a real constant and $\delta_0$ is the Dirac delta, in dimensions higher than one. In these cases, the expression  $-\Delta+q\delta_0$ is only defined in the distribution sense, so we consider the self-adjoint extensions of the symmetric operator
 $$\mathcal H=-\Delta|_{C^\infty_0(\mathbb{R}^d\setminus\{0\})}.$$
 If $d\geq 4$, $\mathcal H$  is essentially self-adjoint (see \cite{Simon73b}), meaning the only self-adjoint extension is the trivial one. However,  for $d=1,2,3$ the extensions  are  well known  classes of operators (see, for example, \cite{BF61} and \cite{AH81}). 
We focus on dimensions $d=2,3$ where the self-adjoint extensions form a one-parameter family of operators, parametrized by $\alpha\in (-\infty,+\infty]$. When $\alpha=+\infty$, the operator $-\Delta_{+\infty}$ corresponds to the Friedrichs extension of $\mathcal{H}$, the free Laplace operator.

The description of the domain $\mathcal{D}(-\Delta_\alpha)$, the resolvent formula  and the spectral properties are well known. In particular, for every $\alpha\in\mathbb{R}$, there exists $\omega_\alpha\geq0$ such that $-\Delta_\alpha+\omega>0$, for every $\omega>\omega_\alpha$. 
Recent studies (\cite{GMS18} for  $d=3$ and  \cite{GMS24} for $d=2$) have examinated the fractional operators $(-\Delta_\alpha+\omega)^{s/2}$  with $ s \in (0,2)$, providing an explicit characterization of their domains, that are independent of $\omega$.  A direct formula for negative fractional powers $(-\Delta_\alpha+\omega )^{-s/2}$ has been obtained allowing the description of the spaces  $\mathcal{D}((-\Delta_\alpha+\omega )^{s/2})$  in relation to the standard Sobolev spaces  
$$H^s(\mathbb{R}^d)=\{u\in L^2(\mathbb{R^d})|(1-\Delta)^{s/2}u\in L^2(\mathbb{R}^d)\}.$$ 

The similarities between the domains $\mathcal{D}((-\Delta_\alpha+\omega)^{s/2})$ and classical Sobolev spaces have led to the introduction of the name singular-perturbed Sobolev spaces $H^s_\alpha(\mathbb{R}^d)$. 
 
In our previous work \cite{GR25}, we extended the definition of  $H^1_\alpha(\mathbb{R}^2)$ to the $L^p$ setting. We considered the fractional operator $(-\Delta_\alpha+\omega )^{-1/2}$ on $L^2(\mathbb{R}^2)\cap L^p(\mathbb{R}^2)$, for $p>1$ and we analyzed the $L^p$-closure of its range. In this way, we defined the spaces  $H^{1,p}_\alpha(\mathbb{R}^2)=\mathcal{D}((-\Delta_\alpha+\omega)^{1/2})\subseteq L^p(\mathbb{R}^2)$, where $(-\Delta_\alpha+\omega)^{1/2}$ is seen as an unbounded operator from $L^p(\mathbb{R}^2)$ into  $L^p(\mathbb{R}^2)$. We also generalized the Strichartz estimates obtained in \cite{CMY19,CMY19b}, extending them to $H^{1,p}_\alpha(\mathbb{R}^2)$. Thanks to these tools, we established the local existence of the solution of the Nonlinear Schr\"odinger equation with power-type nonlinearity 
\begin{equation}\label{eq.NLSE}
    ( i \partial_t - \Delta_\alpha) u =\mu u|u|^{p-1}, \ p >1, \mu = \pm 1,
\end{equation}
in the energy space $H^1_\alpha(\mathbb{R}^2)$.

The present paper aims to complete the description of the fractional perturbed Sobolev spaces $H^{s,p}_\alpha(\mathbb{R}^d)=\mathcal{D}((-\Delta_\alpha+\omega)^{s/2})\subseteq L^p(\mathbb{R}^d)$ for both dimensions $d=2,3$, for all $s\in (0,2)$  and for suitable $p$, and to show analogies with the classical fractional Sobolev spaces
\begin{equation} \label{eq.Sobolev spaces}
    H^{s,p}(\mathbb{R}^d)=\{u\in L^p(\mathbb{R}^d)|(1-\Delta)^{s/2}u\in L^p(\mathbb{R}^d)\}.
\end{equation}
We analyze the properties of the fractional derivatives $(-\Delta)^{s/2}\mathbb{G}_\omega$ of the Green function and  we  generalize and simplify the proofs in \cite{GR25}. We also state in Section \ref{s.Sobolev embedding} a new Sobolev embedding inequality, that holds for perturbed spaces in dimension $d=2$. The precise descriptions of $H^{s,p}_\alpha(\mathbb{R}^d)$  and the extensions of Strichartz estimates to $H^{s,p}_\alpha(\mathbb{R}^d)$ allow to treat the local existence of the solution of the Nonlinear Schr\"odinger equation also in dimension $d=3$. The main novelty is the local well-posedness in $H^s_\alpha(\mathbb{R}^3)$ for suitable $s<1$, with the solution map being uniformly continuous.

The key estimate we use in this local existence result is
\begin{equation}
        \begin{aligned}
            \left\| \phi|\phi|^{p-1}\right\|_{H^{s,\ell}(\mathbb{R}^3)}\leq C \|\phi\|^p_{H^{s,\ell_1}_\alpha(\mathbb{R}^3)}
        \end{aligned}
    \end{equation}  
for $p+s<2$ and $\ell \in(3/2,2]$, $\ell_1\in[2,3)$ under certain conditions, as stated later in Corollary \ref{c.K2}.

\subsection{Domain of $\Delta_\alpha$}
The domain and action  of $\Delta_\alpha$ are well known and can be written in an explicit way ({ see} for example \cite{AH81}):
\begin{gather}
\label{eq.op_dom}\mathcal{D}(-\Delta_\alpha)\;=\;\Big\{\phi\in L^2(\mathbb{R}^d)\,\Big|\,\phi=g_{\alpha,\omega}+\frac{ g_{\alpha,\omega}(0)}{\alpha+c(\omega)}\,\mathbb{G}_\omega\textrm{ with }g_{\alpha,\omega}\in H^2(\mathbb{R}^d)\Big\},\\
\label{eq.opaction}
 (-\Delta_\alpha+\omega)\,\phi\;=\;(-\Delta+\omega)\,g_{\alpha,\omega}\,,
\end{gather}
where $\omega$ is a fixed complex number such that:
\begin{equation}\label{eq.2d1}
\begin{aligned}
& \omega \in \mathbb{C} \setminus \sigma(\Delta) = \mathbb{C} \setminus (-\infty,0], \\
    &  \alpha + c(\omega) \neq 0.
\end{aligned}
\end{equation}
This representation is independent of $\omega$.

The  function $\mathbb{G}_\omega$ is the unique {$L^2$-}solution of the Helmholtz equation for the Dirac delta
\begin{equation}\label{eq.HelmholtzDelta}
    (\omega-\Delta) \mathbb{G}_\omega  = {\delta_0},
\end{equation}
and $c(\omega)$ is the zeroth order term of the Taylor expansion of $\mathbb{G}_\omega$ near zero.
To be more precise, we refer to
\begin{equation} \label{eq.defGlambda}
\mathbb{G}_\omega(x) =
\left\{\begin{aligned}
 & \frac{1}{2\pi} K_{0}(\sqrt{\omega} |x|), \ \ \mbox{if $d=2$,} \\
 &  \frac{e^{-\sqrt{\omega}|x|}}{4\pi|x|}, \ \ \ \mbox{if $d=3$,}
\end{aligned}\right.
\end{equation}
where $K_{0}$ is  the modified Bessel function of order zero. The constant $c(\omega)$ is given by
\begin{equation}\label{eq.defc}
     c(\omega) =  \left\{
     \begin{aligned}
      & \frac{\gamma -\ln 2}{2\pi} + \frac{1}{4\pi} \ln \left( \omega \right)  \ \ \mbox{if $d=2,$} \\
       & \frac{\sqrt{\omega}}{4\pi} \ \ \mbox{if $d=3,$}
     \end{aligned}\right.
    \end{equation} 
where $\gamma\simeq0.577\dots$ is the Euler-Mascheroni constant and $\sqrt{\omega}=e^{\ln{(\omega)}/2}$ denotes the complex square root.

The unique root of $\alpha+c(\omega)=0$, if it exists, is real, positive and we denote it by $E_\alpha$. Explicitly, we have
\begin{equation} E_\alpha =
  \left\{  \begin{aligned}
        &  4 e^{-4\pi \alpha -2\gamma}   &\mbox{if $d=2$,} \\
     &(4\pi \alpha)^2  & \mbox{if $d=3$ and $\alpha <0,$}\\
     &  \mbox{does not exist} & \mbox{ if $d=3$ and $\alpha \geq 0$.}
    \end{aligned}\right.
\end{equation}
The number $E_\alpha$, when it exists, is the unique eigenvalue of $\Delta_\alpha$, with corresponding normalized eigenfunction
$$\psi_\alpha=\frac{\mathbb{G}_{E_\alpha}}{\|\mathbb{G}_{E_\alpha}\|_{L^2(\mathbb{R}^d)}}.$$
In this case, we have the explicit structure of the spectrum
\begin{equation}
    \sigma(\Delta_\alpha) =\sigma_{ess}(\Delta_\alpha)\cup\sigma_p(\Delta_\alpha)= (-\infty,0] \cup  \{E_\alpha\}.
\end{equation}
Otherwise, if $d=3$ and $\alpha \geq 0$, the spectrum is only essential
\begin{equation}
    \sigma(\Delta_\alpha) =\sigma_{ess}(\Delta_\alpha)= (-\infty,0].
\end{equation}

\subsection{Overview on existing results}
We give here a more detailed description of the results concerning the Sobolev Spaces $H^s_\alpha(\mathbb{R}^d)$, the Strichartz estimates and the local well-posedness of \eqref{eq.NLSE}.
To obtain the dispersive estimates without weights, an approach via wave operators was used.
In  \cite{CMY19,CMY19b} for $d=2$  and \cite{DMSY18} for $d=3$ 
the wave operators associated to the couple $(\-\Delta_\alpha,-\Delta)$
$$W_{\alpha}^\pm=\lim_{t\to \pm\infty}e^{it\Delta_{\alpha}}e^{-it\Delta}$$ were studied. It was proved that they  are bounded in $L^p(\mathbb{R}^2)$ for every $p>1$, while on dimension $d=3$ we have some smaller range on $p$, they are bounded on $L^p(\mathbb R^3)$ for $1<p<3$ and unbounded otherwise.
This bound on $p$ for dimension $d=3$ will be reflected on all following results.

 Direct consequences are the  $L^p-L^{p^\prime}$ estimates  
\begin{equation}\label{eq.IS}
      \|e^{it\Delta_{\alpha}}P_{ac}u\|_{L^{p}(\mathbb R^d)}\leq C_p t^{-\frac{d}{2}(\frac{1}{2}-\frac{1}{p})}\|u\|_{L^{p^\prime}(\mathbb R^d)}
\end{equation}
 for every $p\in [2,\infty)$ for $d=2$ and for every $p\in [2,3)$ for $d=3$, { with $p^\prime$ the H\"older conjugate of $p$, i.e. $\frac{1}{p}+\frac{1}{p^\prime}=1$.
 Further,  Strichartz estimates are obtained}:
\begin{equation}\label{eq.Strichartz}
    \begin{aligned}
& \left\| e^{i t \Delta_\alpha} P_{ac}f \right\|_{L^q(\mathbb{R},L^p(\mathbb{R}^d))} \lesssim \|f\|_{L^2(\mathbb{R}^2)}, \\
& \left\|\int_0^t e^{i(t-\tau)\Delta_\alpha} P_{ac}F(\tau) d\tau\right\|_{L^q(\mathbb{R},L^p(\mathbb{R}^d))} \lesssim  \left\| F\right\|_{L^{{s}'}(\mathbb{R},L^{{r}'}(\mathbb{R}^d)) },
    \end{aligned}
\end{equation}
where $P_{ac}$ denotes the projection on the absolutely continuous space and the couples $(p,q)$ and $(s,r)$ are Strichartz exponents, i.e.
\begin{equation}
    \frac{2}{q} + \frac{d}{p} = \frac{d}{2}, \mbox{ with } \left\{\begin{aligned}
        &p\in[2,\infty),\ q\in(2,\infty], &d=2,\\
        &p\in[2,3),\ q\in (4,\infty], &d=3.
    \end{aligned}\right.
\end{equation}
The explicit structure of the absolutely continuous subspace for $-\Delta_\alpha$ allows us to generalize the above inequalities locally in time, without the orthogonal projection.

In \cite{GMS18} for $d=3$ and in \cite{GMS24} for $d=2$ there is also an explicit characterization of fractional domains $\mathcal{D}((\omega-\Delta_\alpha)^\frac{s}{2})=H^s_\alpha(\mathbb{R}^d)$.
 They showed the following equalities for $s\in [0,2]$
 \begin{equation}\label{eq.Hs}
     \begin{aligned}
         &H^s_\alpha(\mathbb{R}^2)=\left\{\begin{aligned}
             &H^s(\mathbb{R}^2) &s<1,\\
             &\{\phi\in L^2(\mathbb{R}^2)|\phi=g+{\frac{g(0) }{\alpha + c (\omega)}} \mathbb{G}_\omega, g\in H^s(\mathbb{R}^2)\} &s>1,
         \end{aligned}\right.\\
         &H^s_\alpha(\mathbb{R}^3)=\left\{\begin{aligned}
             &H^s(\mathbb{R}^3) &s<\frac{1}{2},\\
             &H^s(\mathbb{R}^3)\dot+\operatorname{Span}\{\mathbb{G}_\omega\} &\frac{1}{2}<s<\frac{3}{2},\\
             &\{\phi\in L^2(\mathbb{R}^3)|\phi=g+{\frac{g(0) }{\alpha + c (\omega)}}\mathbb{G}_\omega, g\in H^s(\mathbb{R}^3)\} &s>\frac{3}{2},
         \end{aligned}\right.
     \end{aligned}
 \end{equation}
{ where the sets on the right do not depend on $\omega$}. Here, we did not write the transition cases, that are a bit delicate.
In particular, they proved that for high $s$ the spaces $H^s_\alpha(\mathbb{R}^d)$ have a structure similar to $\mathcal{D}(\Delta_\alpha)$; their elements can be decomposed in a regular and a singular parts, while for low $s$, they coincide with standard Sobolev spaces.

{
In \cite{GR25}, we proved a characterization  for $H_\alpha^{1,p}(\mathbb{R}^2)=\mathcal{D}((-\Delta_\alpha+\omega)^{1/2})\subseteq L^p(\mathbb{R}^2)$ that is perfectly consistent to \eqref{eq.Hs}:
\begin{equation}\label{eq.H1p}
    \begin{aligned}
         &H^{1,p}_\alpha(\mathbb{R}^2)=\left\{\begin{aligned}
             &H^{1,p}(\mathbb{R}^2) &p<2,\\
             &\{\phi\in L^p(\mathbb{R}^2)|\phi=g+{\frac{g(0) }{\alpha + c (\omega)}}\mathbb{G}_\omega, g\in H^s(\mathbb{R}^2)\} &p>2,
         \end{aligned}\right.
    \end{aligned}
\end{equation}
 with the sets on the right non depending on $\omega$ and we extended \eqref{eq.Strichartz} to
\begin{equation}\label{eq.Strichartz2}
    \begin{aligned}
& \left\| e^{i t \Delta_\alpha} f \right\|_{L^q((0,T)H^{1,p}_\alpha(\mathbb{R}^2))} \lesssim \|f\|_{L^2(\mathbb{R}^2)}, \\
& \left\|\int_0^t e^{i(t-\tau)\Delta_\alpha} P_{ac}F(\tau) d\tau\right\|_{L^q((0,T),H^{1,p}_\alpha(\mathbb{R}^2))} \lesssim  \left\| F\right\|_{L^{{s}'}((0,T),H^{1,{r}'}_\alpha(\mathbb{R}^2)) }.
    \end{aligned}
\end{equation}
}

In \cite{CFN21}  well-posedness of the nonlinear  Schr\"odinger equation \eqref{eq.NLSE} in dimensions $d=2,3$ was treated.
They proved that for $d=2$ and $p\geq1$, $d=3$ and $1\leq p<\frac{3}{2}$, there exists $T\in(0,1]$ such that \eqref{eq.NLSE} is well-posed in
$$C([0,T];\mathcal{D}(\Delta_\alpha))\cap C^1([0,T];L^2).$$
Moreover for $1<p<3$ if $d=2$ (the subcritical case), and for $1< p<\frac{3}{2}$ if $d=3$, the solution is global. The method used in \cite{CFN21} is the Kato method (see  \cite{K95}).
The existence of standing waves for the  2d Hartree type equation with point interaction   is  studied in \cite{GMS24}
\begin{equation}
    i\partial_tu=-\Delta_\alpha u +(w*|u|^2)u,
\end{equation}
where $w$ is a real-valued measurable function. 
The local existence result in this work is obtained using the method of \cite{OSY12}, which is an improvement of Cazenave's regularization approach that provides existence, uniqueness and conservation of mass, energy  of the solution, but there is no information about  regularity of the solution map. Recall that Kato method gives positive answers to the last point (see \cite{K95}). In the work \cite{FN23}, where they also study blow-up, the local existence theorem of { \cite{FGI22} }is used.

A local existence result in 3d case is treated  in \cite{MOS18}, where the local well-posedness is established in $H^s_\alpha{(\mathbb{R}^3)}$ for $ s \in [0,3/2),$ $ s \neq 1/2$ and in the radial $H^s_\alpha{(\mathbb{R}^3)}$ space for $s \in (3/2,2].$

{ In \cite{GR25}, we proved the local existence for the solution of \eqref{eq.NLSE} for the mass critical and supercritical cases $(p\geq3)$ in the energy space $H^1_\alpha(\mathbb{R}^2)$ and in $L^2(\mathbb{R}^2)$ for the mass subcritical case ($p<3$).

Also stability and instability of standing waves associated to \eqref{eq.NLSE} are studied and they require well-posedness results in energy space. Their existence is stated in \cite{ABCT} and \cite{ABCT22}.  
In \cite{FGI22}, treating the stability for $d=2$,   the existence of the solution map in the space $H^1_\alpha(\mathbb{R}^2)$ is proved, by appropriately modifying of Cazenave's approach and using a compactness argument.

In \cite{BGR25}, local and global well-posedness of solution of the reaction-diffusion equation with point interaction in dimension $d=2$ are studied. In \cite{FG24}, a new proof of local well-posedness of \eqref{eq.NLSE}, with initial datum  in $H^1_\alpha(\mathbb{R}^2)$ is provided, including the blow-up alternative. The Kato's method is used, estimating the $H^{1,p}_\alpha$-norm of the nonlinearity. 

\subsection{Organization of the paper}  
\begin{itemize}
    \item Section \ref{s.Main Results} presents the main results of this paper.  
\item In Section \ref{s.Preliminary}, we recall some preliminary results that will be used throughout the paper.
\item Section \ref{s.Lp} is devoted to extending the definition of $\Delta_\alpha$ as an operator on $L^p(\mathbb{R}^d)$.  

\item In Section \ref{s.Sobolev embedding}, we establish a Sobolev inequality for the spaces $H^{s,p}_{\alpha}(\mathbb{R}^d)$. 
\item Section \ref{s.characterization} provides an explicit characterization of these spaces.  

\item Finally, in Section \ref{s.lwp}, we prove the local well-posedness results. 
\end{itemize}

\section{Main Results}\label{s.Main Results}

The first result of this paper is the characterization of the spaces $H^{s,p}_\alpha(\mathbb{R}^d)=\mathcal{D}((\omega-\Delta_\alpha)^{s/2})\subseteq L^p(\mathbb{R}^d)$ in a similar way to \eqref{eq.Hs} and \eqref{eq.H1p}. See Subsection \ref{subsection Hspa} for a more precise definition of $H^{s,p}_\alpha(\mathbb{R}^d)$.
From now on $p$ will be in these intervals
\begin{equation} \label{eq. condition p}
        p \in \left\{
\begin{aligned}
    & (1,\infty) &\mbox{if $d=2,$}\\
    & \left(\frac{3}{2},3\right) &\mbox{if $d=3.$}
\end{aligned}
        \right.
    \end{equation}
We will see that, also in the general case, for every fixed $p$ that satisfies \eqref{eq. condition p}, for small $s$ the spaces $H^{s,p}_\alpha(\mathbb{R}^d)$ coincide with the classical one $H^{s,p}(\mathbb{R}^d)$, because the singularity is not strong enough. Otherwise every function in $H^{s,p}_\alpha(\mathbb{R}^d)$ can be decomposed in a regular part in $H^{s,p}(\mathbb{R}^d)$ and a singular part. We will state this result in the three following next Theorems.

\begin{Theorem} \label{t.2.2}
  Let $\alpha\in \mathbb{R}$, $d=2,3$, $s\in [0,2]$ and $p>1$ for $d=2$ or $\frac{3}{2}<p<3$ for $d=3$. If $s < \frac{d}{p} - d + 2,$  then
  \begin{equation}
       H^{s,p}_\alpha(\mathbb{R}^d) = H^{s,p}(\mathbb{R}^d).
  \end{equation}
\end{Theorem}

\begin{Theorem} \label{t.23A1} 
Let $\alpha\in \mathbb{R}$, $d=3$, $s\in [0,2]$ and $\frac{3}{2}<p<3$. If $\frac{3}{p}-1 < s < \frac{3}{p},$  then for any $\omega \in \mathbb{C} \setminus \sigma(\Delta_\alpha)$   we have
  \begin{equation}\label{eq.mrpg66}
  \begin{aligned}
     &  H^{s,p}_\alpha(\mathbb{R}^3) = H^{s,p}(\mathbb{R}^3) \dot+  \mathrm{span} \left\{ \mathbb{G}_\omega \right\}.
  \end{aligned}
      \end{equation}
\end{Theorem}

\begin{Remark}
    We note that the range of $s$ in Theorem \ref{t.23A1} coincides with the range
    $$ \frac{d}{p}-d+2 < s < \frac{d}{p},$$
    that is empty for $d=2$. For this reason, the theorem is formulated only in dimension $d=3.$
\end{Remark}

\begin{Theorem} \label{t.2.1}
Let $\alpha\in \mathbb{R}$, $d=2,3$, $s\in [0,2]$ and $p>1$ for $d=2$ or $\frac{3}{2}<p<3$ for $d=3$. If $s>\frac{d}{p},$ then for any $\omega \in \mathbb{C} \setminus \sigma(\Delta_\alpha)$   we have
  \begin{equation}\label{eq.mrpg2}
  \begin{aligned}
     &   H^{s,p}_\alpha(\mathbb{R}^d) =  \left\{ \phi = g + \frac{g(0) }{\alpha + c (\omega)} \mathbb G_{\omega}, \ \ g \in  H^{s,p}(\mathbb{R}^d)   \right\},
  \end{aligned}
      \end{equation}
      with $c(\omega)$ defined in \eqref{eq.defc}.
\end{Theorem}

Thanks to these spaces we will treat the following Cauchy problem
\begin{equation}\label{eq.CP81}
\begin{aligned}
    &( i \partial_t + \Delta_\alpha) u =\mu u|u|^{p-1}, \ p >1, \mu = \pm 1,\\
   & u(0) = u_0 \in H^s_\alpha(\mathbb{R}^d).
    \end{aligned}
\end{equation}

The mass and energy are conserved
\begin{equation}
     \|u(t)\|^2_{L^{2}(\mathbb{R}^d)}  =  \|u(0)\|^2_{L^{2}(\mathbb{R}^d)}  , \ \ E(t) = E(0),
\end{equation}
with
\begin{equation}
\begin{aligned}
    & E(t) =   \frac{1}{2} \langle -\Delta_\alpha u(t), u(t) \rangle_{L^2(\mathbb{R}^d)} + \frac{\mu}{p+1} \|u(t)\|^{p+1}_{L^{p+1}(\mathbb{R}^d)}  \\
    & =\frac{1}{2} \left\| (\omega-\Delta_\alpha)^{1/2} u(t)\right\|^2_{L^2(\mathbb{R}^d)} -  \frac{\omega}{2} \|u(t)\|^2_{L^{2}(\mathbb{R}^d)}   + \frac{\mu}{p+1} \|u(t)\|^{p+1}_{L^{p+1}(\mathbb{R}^d)}.
\end{aligned}
\end{equation}
First we prove the local existence result in $L^2(\mathbb{R}^3)$ with $p\in[1,2)$. We note that the range of $p$ is slightly wider than $[1,\frac{3}{2})$, that is the interval for local existence  in $H^2_\alpha(\mathbb{R}^3)$ (see \cite{CFN21}).
 
\begin{Theorem}\label{t.le19}
   Let $d=3$. For any $p \in[1,2)$ and any $R>0$ there  exists $T=T(R,p)>0$ so that for any
    $$ u_0 \in B_{L^2}(R) = \left\{ \phi \in L^2(\mathbb{R}^3);  \|\phi\|_{L^2(\mathbb{R}^3)} \leq R \right\}$$
    there exists a unique solution
    $$ u \in C([0,T]; L^2(\mathbb{R}^3))$$
    to the integral equation
\begin{equation} \label{eq.ie30}
    u = e^{it \Delta_\alpha} u_0 -i{\mu} \int_0^t e^{i(t-\tau)\Delta_\alpha}  u(\tau)|u(\tau)|^{p-1} d\tau
\end{equation}
    associated to \eqref{eq.CP81}.
\end{Theorem}
We now state local existence in the spaces $H^s_\alpha(\mathbb{R}^d)$. We resume the following result in dimension $d=2$, already proved in \cite{GR25} and we improve it also for small $p\in (1,2)$.
\begin{Theorem}\label{t.Schrodinger2}
    Let $d=2$, for any $p>1$ and any $R>0$ there  exists $T=T(R,p)>0$ so that for any
    $$ u_0 \in B(R) = \left\{ \phi \in H^1_\alpha(\mathbb{R}^2); \|\phi\|_{H^1_\alpha(\mathbb{R}^2)} \leq R \right\}$$
    there exists a unique mild solution
    $$ u \in C([0,T]; H^1_\alpha(\mathbb{R}^2))$$
    to the integral equation
\begin{equation}
    u = e^{it \Delta_\alpha} u_0 -i {\mu}\int_0^t e^{i(t-\tau)\Delta_\alpha}  u(\tau)|u(\tau)|^{p-1} d\tau
\end{equation}
    associated to \eqref{eq.CP81}.  The solution map
    $$ u_0 \in B(R)  \mapsto u \in C([0,T]; H^1_\alpha(\mathbb{R}^2)) $$
    is  continuous for $1 < p \leq 2,$ for $p>2$ the solution map is uniformly Lipschitz continuous.
\end{Theorem}

Finally we will prove a similar result also in dimension $d=3$ for small $p$ and $s$.

\begin{Theorem}\label{t.Schrodinger3}
    Let $d=3,$   $p \in (1,2)$ and $s \in (0,1)$ satisfy
      \begin{equation}\label{eq.ps54}
      \begin{aligned}
         &  p+s < 2.
      \end{aligned}
      \end{equation}
    
      Then for any  $R>0$ there  exists $T=T(R,p)>0$ so that for any
    $$ u_0 \in B_s(R) = \left\{ \phi \in H^s_\alpha(\mathbb{R}^3); \|\phi\|_{H^s_\alpha(\mathbb{R}^3)} \leq R \right\}$$
    there exists a unique mild solution
    $$ u \in C([0,T]; H^s_\alpha(\mathbb{R}^3))$$
    to the integral equation
\begin{equation}
    u = e^{it \Delta_\alpha} u_0 -i {\mu}\int_0^t e^{i(t-\tau)\Delta_\alpha}  u(\tau)|u(\tau)|^{p-1} d\tau
\end{equation}
    associated to \eqref{eq.CP81}.  
\end{Theorem}

}

\section{Preliminary results}\label{s.Preliminary}
{This section is dedicated to properties and inequalities regarding the Green function $\mathbb{G}_\omega$ defined in \eqref{eq.defGlambda}, the independence of the domain $\mathcal{D}(-\Delta_\alpha)$ from the parameter $\omega$. Finally we define fractional operators on Banach spaces.} 

Owing to \eqref{eq.HelmholtzDelta}, combined with the Fourier transform it can be seen that the regularity
of $ \mathbb G_{\omega}$ is weaker than $H^2$
\begin{equation}\label{eq.regularityG L2}
  \mathbb G_{\omega} \in
    H^{\frac{d}{2}-d+2-\varepsilon} \setminus H^{\frac{d}{2}-d+2}, \ d = 2,3
\end{equation}
for any $\varepsilon \in (0,1].$ We note that both terms in the decomposition of $\phi$ in \eqref{eq.op_dom} are in $L^2$.

The results in \cite{AH81} guarantee the following statements:
\begin{itemize}
  \item The set $\mathcal{D}(-\Delta_\alpha)$ is exactly the domain of the operator:
  \begin{equation}\label{eq.h26}
      \mathcal{D}(-\Delta_\alpha) =  \left\{ \phi \in L^2; \phi = (\omega-\Delta_\alpha)^{-1}f, f \in L^2{(\mathbb{R}^d)}   \right\};
  \end{equation}
    \item the operator $\Delta_\alpha$ is self-adjoint, its spectrum consists of absolutely continuous part $(-\infty,0]$ and it has point eigenvalue at $\omega_0$ determined by $\alpha + c(\omega_0)=0;$
        \item the domain and the action are independent of the choice of $\omega,$ satisfying \eqref{eq.2d1};
    \item the resolvent identity  is explicit :
\begin{equation}\label{eq.res_formula}
(-\Delta_\alpha+\omega)^{-1} f\;=\;(-\Delta+\omega)^{-1}f + \frac{1}{{\alpha}+c(\omega) }\mathbb{G}_\omega \langle f, \mathbb{G}_\omega \rangle ,\end{equation}
 where $\langle f, g \rangle$ denotes the standard inner product in $L^2$ i.e.
\begin{equation}
    \langle f, g \rangle = \int_{\mathbb{R}^2} f(x) \overline{g(x)} dx.
\end{equation}
\end{itemize}
Identity \eqref{eq.res_formula} says that the resolvent of $-\Delta_\alpha$ is a rank-one perturbation of the free resolvent.

We give  the proof { of the} third point for completeness. Fixing for the moment the parameters $\omega$  in \eqref{eq.op_dom} it is easy to see that $g$ is unique.

\begin{Proposition}\label{l.wun0}
    If $\omega$ satisfies \eqref{eq.2d1} and $\phi \in \mathcal{D}(-\Delta_\alpha)$ has two representations
    \begin{equation} \label{phdd32}
    \begin{aligned}
        & \phi=g_1+C_1\,\mathbb{G}_{\omega}, \\
         & \phi=g_2+C_2\,\mathbb{G}_{\omega},
    \end{aligned}
\end{equation}
where $g_1,g_2 \in H^{2}$ and $C_1,C_2 $ are complex constants, then
\begin{equation}
   g_1=g_2,
\end{equation}
and in particular $g_1(0)=g_2(0).$
\end{Proposition}

\begin{proof}
    Assume \eqref{phdd32} holds.
    Then we can write
    $$g_1-g_2={(C_2-C_1)}\mathbb{G}_\omega.$$
    Now we use \eqref{eq.regularityG L2} and the fact that $g_1-g_2\in H^2{(\mathbb{R}^d)}$, but {$\mathbb{G}_\omega\notin H^{2}(\mathbb{R}^d)$}, so we get $ C_1=C_2$ and $g_1=g_2$.
\end{proof}
For the next {proposition} we need a precise asymptotic expansion of $\mathbb{G}_\omega$ near zero:
\begin{equation}\label{eq.taylor0}
    \mathbb{G}_\omega(x)=
    \left\{\begin{aligned}
        &-\frac{1}{2\pi}\ln(|x|)-c(\omega)+o(|x|), &d=2,\\
        &\frac{1}{4\pi|x|}-c(\omega)+O(|x|), &d=3,
    \end{aligned}
    \right. \ \text{ for } x\to 0,
\end{equation}
{with $c(\omega)$ defined in \eqref{eq.defc},} and the fact that the difference of two Green functions is more regular. Thanks to \eqref{eq.HelmholtzDelta}, we have indeed
\begin{equation}\label{eq.diff2G}
    (\omega_1-\Delta)(\mathbb{G}_{\omega_1}-\mathbb{G}_{\omega_2})={(\omega_2-\omega_1)}\mathbb{G}_{\omega_2},
\end{equation}
that, combined with \eqref{eq.regularityG L2}, gives
\begin{equation}\label{eq.regularityG-G L2}
  \mathbb G_{\omega_1}-\mathbb{G}_{\omega_2} \in
    H^{\frac{d}{2}-d+4-\varepsilon} \setminus H^{\frac{d}{2}-d+4}, \ d = 2,3.
\end{equation}
In particular the difference of two Green functions $ \mathbb G_{\omega_1}-\mathbb{G}_{\omega_2}\in H^2$ and can be absorbed in the regular part.
\begin{Proposition}\label{l.ind2}
    The definition \eqref{eq.op_dom} of the operator domain $\mathcal{D}(-\Delta_\alpha)$  as well the definition \eqref{eq.opaction} of the operator action of $-\Delta_\alpha$ are independent of $\omega$ satisfying \eqref{eq.2d1}.
\end{Proposition}
\begin{proof}
    Let $\phi$ be in the domain of $-\Delta_\alpha$,  so $\phi$ has  representation 
    \begin{equation}\label{eq.g1g29}
        \begin{aligned}
           & \phi = g_1 + \frac{g_{1}(0)}{\alpha+c(\omega_1)} \mathbb{G}_{\omega_1} , \  g_{1} \in H^2{(\mathbb{R}^d)}.
        \end{aligned}
        \end{equation}

    Let $\omega_2\neq\omega_1$ that satisfies \eqref{eq.2d1},
    then we consider the function 
    $$g_2=g_1+\frac{g_1(0)}{\alpha+c(\omega_1)}(\mathbb{G}_{\omega_1}-\mathbb{G}_{\omega_2})\in H^2{(\mathbb{R}^d)}$$
    that also satisfies
    \begin{equation}\label{eq.g2(0)}
        g_2(0)=g_1(0)+\frac{g_1(0)}{\alpha+c(\omega_1)}(-c(\omega_1)+c(\omega_2))=\frac{\alpha+c(\omega_2)}{\alpha+c(\omega_1)}g_1(0).
    \end{equation}
    Adding and subtracting $\frac{g_{1}(0)}{\alpha+c(\omega_1)} \mathbb{G}_{\omega_2}$ to \eqref{eq.g1g29}, it is immediate to see that we have the alternative representation of $\phi$
    \begin{equation}
        \begin{aligned}
           & \phi = g_2 + \frac{g_{2}(0)}{\alpha+c(\omega_2)} \mathbb{G}_{\omega_2} , \  g_{2} \in H^2{(\mathbb{R}^d)}.
        \end{aligned}
        \end{equation}
    The representation works also for the action formula \eqref{eq.opaction}. Using, \eqref{eq.g2(0)} and \eqref{eq.diff2G}, we have indeed:
    \begin{equation}
    \begin{aligned}
         (\omega_2-\Delta_\alpha)\phi&=(\omega_2-\omega_1)\phi+(\omega_1-\Delta)g_1\\
         &=\frac{g_{1}(0)}{\alpha+c(\omega_1)}(\omega_2-\omega_1)\mathbb{G}_{\omega_1}+(\omega_2-\Delta)g_1\\
         &=(\omega_2-\Delta)\left(\frac{g_{2}(0)}{\alpha+c(\omega_2)}(\mathbb{G}_{\omega_1}-\mathbb{G}_{\omega_2})+g_1\right)=(\omega_2-\Delta)g_2.
    \end{aligned} 
    \end{equation}
\end{proof}
Now we end this subsection with the following.
\begin{Proposition}
    For any $z \in \mathbb{C} \setminus (-\infty,0]$ the following properties are equivalent{.}
    \begin{enumerate}
        \item[i)] $\mathbb G_{z} \in \mathcal{D} (-\Delta_\alpha);$
    \item[ii)]  $\alpha+c(z)=0.$
    \end{enumerate}
    In this case,  we have 
    \begin{equation}
        (z-\Delta_\alpha)\mathbb{G}_{z}=0.
    \end{equation}
\end{Proposition}
\begin{proof}
{ We first prove i)$\implies$ ii).}     By contradiction, we suppose $\alpha+c(z)\neq0$. The independence of $z$ allows us to decompose
    $$\mathbb G_{z} =  g + \frac{g(0)}{\alpha+c(z)} \mathbb G_{z},$$
with a certain $g\in H^2$, that can be rewritten as
$$\left(1-\frac{g(0)}{\alpha+c(z)}\right)\mathbb G_{z}=g.$$
Since the function $\mathbb G_{z}\notin H^2{(\mathbb{R}^d)}$, this implies $g\equiv0$ and $1-\frac{g(0)}{\alpha+c(z)}=0$, but this is a contradiction because $g$ is continuous.

Otherwise, if $\alpha=-c(z)$, we want to find a certain $g\in H^2{(\mathbb{R}^d)}$ such that
$$\mathbb{G}_z=g+\frac{g(0)}{\alpha+c(\omega)}\mathbb{G}_\omega,$$
for a fixed $\omega$ that satisfies \eqref{eq.2d1}{.}
It is sufficient to consider $g=\mathbb{G}_z-\mathbb{G}_\omega$, because
$$g(0)=-c(z)+c(\omega)=\alpha+c(\omega).$$

We can finally compute the action of $-\Delta_\alpha$ on $\mathbb{G}_z$, thanks to \eqref{eq.diff2G}
\begin{equation}
    (z-\Delta_\alpha)\mathbb{G}_z=(z-\omega)\mathbb{G}_z+(\omega-\Delta)(\mathbb{G}_z-\mathbb{G}_\omega)=0.
\end{equation}
\end{proof}
\subsection{Fractional powers of operators } \label{s.def fractional}
For this subsection, we refer to the papers of Komatsu \cite{Ko66} and \cite{Ko67}.

He considered  closed linear operators $A$ on a Banach space $X$, such that their resolvent set  contains  $(-\infty,0)$  and 
\begin{equation}\label{eq.conditionKomatsu}
    \|\lambda(\lambda+A)^{-1}\|_{\mathcal{L}(X)}\leq M,\ \ \lambda>0
\end{equation}
with a constant $M>0$ independent of $\lambda$. { We denoted with $\|\cdot\|_{\mathcal{L}(X)}$ the operator norm.} In this case he proved the following formula (see  (1.3) in \cite{Ko66})

\begin{equation}\label{eq.fractionalKomatsu}
\begin{aligned}
  &  A^{\sigma} \phi = -\frac{\sin (\pi \sigma)}{\pi} \left(\int_0^R \lambda^\sigma (\lambda+A)^{-1}\phi d\lambda {-} 
      \frac{R^\sigma\phi}{\sigma} {-} \int_R^\infty \lambda^{\sigma-1} A(\lambda+A)^{-1} \phi d\lambda \right),
\end{aligned}
\end{equation}
with $\mathrm{Re} \sigma \in (-1,1)$ and for any $R>0$. We note that both integrals are well defined and finite thanks to \eqref{eq.conditionKomatsu}. His formula is more general and involves also the wider range ${\mathrm{Re} \sigma\in}(-n,n)$ for any $n\in \mathbb{N}$, but in this paper we only need  the case $n=1$.  
We can also simplify the fractional formula if we distinguish between positive and negative powers.

If $\mathrm{Re}\sigma\in(0,1)$, computing the limits $R\to0$ and $R\to \infty$ respectively, we obtain 
\begin{equation}\label{eq.fractionalKomatsu positive}
\begin{aligned}
  &  A^{\sigma} \phi = \frac{\sin (\pi \sigma)}{\pi} \int_0^\infty \lambda^{\sigma-1} A(\lambda+A)^{-1} \phi d\lambda ,
\end{aligned}
\end{equation}
and
\begin{equation}\label{eq.fractionalKomatsu negative}
\begin{aligned}
  &  A^{-\sigma} \phi = \frac{\sin (\pi \sigma)}{\pi} \int_0^\infty\lambda^{-\sigma} (\lambda+A)^{-1}\phi d\lambda,
\end{aligned}
\end{equation}
that coincide with the formulas for self-adjoint operators from spectral theorem (see for example Section 5, point C in \cite{EN2000}) .
We will resume in a Proposition the results that we will use later.
\begin{Proposition}\label{p.Komatsu}
    { Let $\alpha,\beta\in \mathbb{C}$, let $x\in\mathcal{D}(A^\beta)\cap\mathcal{D}(A^{\alpha+\beta})$. Then $A^\beta x\in \mathcal{D}(A^\alpha)$ and satisfies
    \begin{equation}
        A^\alpha A^\beta x= A^{\alpha+\beta}x .
    \end{equation}  
Moreover, if $\mathrm{Re} \sigma \in (0,1)$, then $\mathcal{D}(A^{\sigma})$ is contained and dense in $\overline{\mathcal{D}(A)}$ and the range $R(A^\sigma)$ is contained in $\overline{\mathcal{D}(A)}\cap\overline{R(A)}$. Otherwise,  $\mathcal{D}(A^{-\sigma})$ is contained and dense in $\overline{R(A)}$ and the range $R(A^{-\sigma})$ is contained in $\overline{\mathcal{D}(A)}\cap\overline{R(A)}$.}
    
\end{Proposition}

{If we define the Fourier transform in $L^2(\mathbb{R}^d)$ as
\begin{equation}\label{eq.Fourier}
    \mathcal{F}[f](\xi)=\widehat{f}(\xi)=(2\pi)^{-d/2}\int_{\mathbb{R}^d} e^{ix\cdot\xi }f(x)dx
\end{equation}
and the inverse Fourier transform as
\begin{equation}\label{eq.inverse Fourier}
    \mathcal{F}^{-1}[f](x)=\mathcal{F}[f](-x),
\end{equation}
we can see that the fractional Laplacian $(-\Delta)^{s/2}f$ in $L^2(\mathbb{R}^d)$, defined with Komatsu's formula, coincides with the Riesz potential $\mathcal{F}^{-1}[|\xi|^s\widehat{f}]$,
thanks to the spectral theorem applied to the following identity (see Ch. 10.4 in \cite{BS87}).
\begin{equation}
    \begin{aligned}
       & x^{s/2}=\frac{\sin{(\pi s/2)}}{\pi}\int_0^\infty\lambda^{s/2-1}\frac{x}{x+\lambda}d\lambda&x\geq0,\ s\in(0,2).
    \end{aligned}
\end{equation}

}

\subsection{Regularity of the Green function $\mathbb{G}_\omega$} \label{s.regularityG}
In this subsection we will summarize some properties of the Green function $\mathbb{G}_\omega$, the fundamental solution of $(\omega - \Delta)$ for any
\begin{equation}\label{eq.ome1}
    \omega \in \mathbb{C} \setminus \sigma(\Delta) = \mathbb{C} \setminus (-\infty,0].
\end{equation}
First of all, it can be represented, for every $d\geq2$, as
\begin{equation} \label{eq.defGgeneral}
\mathbb{G}_\omega(x) = (2\pi)^{-d/2} \omega^{(d-2)/4} |x|^{-(d-2)/2} K_{(d-2)/2}
(\sqrt{\omega} |x|).
\end{equation}
Here  $K_\nu$ is the second type modified Bessel function of order $\nu \geq 0$, also called Macdonald function. { For $\nu \neq 0$, the modified Bessel function of the second kind admits the following asymptotic expansion as $z \to 0$:
\begin{equation}\label{eq.taylor 0 K}
   \begin{aligned}
       & K_\nu(z)=z^{-\nu}\big(2^{\nu-1}\Gamma(\nu)+o(z)\big)+z^\nu\big(2^{-\nu-1}\Gamma(-\nu)+o(z)\big), &z\to0 ,
   \end{aligned}
\end{equation}
The case $\nu = 0$ is given before in the first line of equation \eqref{eq.taylor0}.
If $\nu<0$, the modified Bessel function is defined as $K_\nu=K_{-\nu}$.
The functions $K_\nu$ decay exponentially when $z\to\infty$, the asymptotic expansions  are given by:
\begin{equation}\label{eq.taylor infinity}
    \begin{aligned}
        &K_\nu(z)=\sqrt{\frac{\pi}{2z}}e^{-z}\big(1+O(z^{-1})\big), &z\to\infty.
    \end{aligned}
\end{equation}
}
We focus only on the cases $d=2,3$ in \eqref{eq.defGgeneral} and we have
$$K_{1/2}({|x|})=\sqrt{\frac{\pi}{2|x|}}e^{-|x|},$$
so, as stated  in  \eqref{eq.defGlambda}, we have the explicit formulas
\begin{equation} \label{eq.defGlambda1}
\mathbb{G}_\omega(x) =
\left\{\begin{aligned}
 & \frac{1}{2\pi} K_{0}(\sqrt{\omega} |x|), \ \ \mbox{if $d=2$,} \\
 &  \frac{e^{-\sqrt{\omega}|x|}}{4\pi|x|}, \ \ \ \mbox{if $d=3$.}
\end{aligned}\right.
\end{equation}
From \eqref{eq.defGgeneral}, we obtain the rescaling property
\begin{equation}\label{eq.rescalingG}
    \mathbb{G}_\omega(x)=\omega^{\frac{d}{2}-1}\mathbb{G}_1(\sqrt{\omega}x).
\end{equation}

Thanks to asymptotic expansions in zero \eqref{eq.taylor0}  and at infinity \eqref{eq.taylor infinity} it is is possible to verify that
\begin{equation}\label{eq.LpGreen}
    \begin{aligned}
        \mathbb{G}_\omega \in L^p(\mathbb{R}^d)  \ \ \mbox{for} \ \ 
       \left\{ \begin{aligned}
            & p \in [1,\infty) \ \ \mbox{if $d=2,$}\\
            &p \in [1,3) \ \ \mbox{if $ d=3,$}
        \end{aligned}\right.
    \end{aligned}
\end{equation}
    that, combined with \eqref{eq.rescalingG}, gives the $L^p$ rescaling
    \begin{equation}
        \|\mathbb{G}_\omega\|_{L^p{(\mathbb{R}^d)}}={|\omega|}^{\frac{d}{2}-1-\frac{d}{2p}}\|\mathbb{G}_1\|_{L^p{(\mathbb{R}^d)}}.
    \end{equation}
    We can resume this in the following Proposition
\begin{Proposition} \label{lemA3}
    If $p$ satisfies
    \begin{equation}\label{eq.conditionp}
        p \in \left\{
\begin{aligned}
    & [1,\infty) &\mbox{if $d=2,$}\\
    & [1,3) &\mbox{if $d=3,$}
\end{aligned}
        \right.
    \end{equation}
    then for any $\omega \in \mathbb{C} \setminus (-\infty,0]$ we have the estimate{}
   \begin{equation}\label{eq.estpg3}
    \left\| \mathbb{G}_\omega\right\|_{L^p(\mathbb{R}^d)} \lesssim |\omega|^{\frac{d}{2}-1-\frac{d}{2p}}
\end{equation} 

\end{Proposition}
    Computing the gradient, the regularity can also be increased 
    \begin{equation}\label{eq.WpGreen}
    \begin{aligned}
        \mathbb{G}_\omega \in {H}^{1,p}(\mathbb{R}^d)  \ \ \mbox{for} \ \ 
       \left\{ \begin{aligned}
            & p \in [1,2)  &\mbox{if $d=2,$}\\
            &p \in [1,3/2)  &\mbox{if $ d=3,$}
        \end{aligned}\right.
    \end{aligned}
\end{equation}
and, as stated in \eqref{eq.regularityG L2}, with Fourier transform, can be verified that
\begin{equation}
  \mathbb G_{\omega} \in
    H^{\frac{d}{2}-d+2-\varepsilon} \setminus H^{\frac{d}{2}-d+2}, \ d = 2,3.
\end{equation}
{The regularity of $\mathbb{G}_1$ can be generalized to the fractional Sobolev spaces $H^{s,p}(\mathbb{R}^d)$ defined in \eqref{eq.Sobolev spaces} thanks to the explicit formulas of $(1-\Delta)^{s/2}\mathbb{G}_1$.
The relations (4.1) and (4.6) in \cite{AS61} compute an explicit formula for
\begin{equation}\label{eq.fractional G transform}
    \begin{aligned}
       & \mathcal{F}^{-1}[(2\pi)^{-d/2}(1+|\xi|^2)^{-\beta/2}]=\mathcal{G}_\beta, 
    \end{aligned}
\end{equation}
with
\begin{equation}\label{eq.fractional G}
   \begin{aligned}
       & \mathcal{G}_\beta(x)=C(d,\beta)K_{\frac{d-\beta}{2}}(|x|)|x|^{\frac{\beta-d}{2}}, &C(d,\beta)=\frac{2^{(2-\beta-d)/2}}{\pi^{d/2}\Gamma(\beta/2)},
   \end{aligned}
\end{equation}
that allows to verify
\begin{equation}\label{eq.explicit fractional G1}
    (1-\Delta)^{s/2}\mathbb{G}_1(x)=\mathcal{F}^{-1}[(2\pi)^{-d/2}(1+|\xi|^2)^{-\frac{2-s}{2}}](x)= C(d,2-s) K_{\frac{d+s-2}{2}}(|x|)|x|^{\frac{2-s-d}{2}}.
\end{equation}
Although the regularity of $\mathbb{G}_1$ can be readily derived from the asymptotic expansions \eqref{eq.taylor0}, \eqref{eq.taylor 0 K}, and \eqref{eq.taylor infinity} within \eqref{eq.explicit fractional G1}, we will instead provide an alternative proof using the fractional Komatsu formula, as this will serve as the main approach throughout the paper.

}
\begin{Lemma}
Let $s\in [0,2]$ and $p$ satisfies \eqref{eq.conditionp}, i.e.
\begin{equation}
        p \in \left\{
\begin{aligned}
    & [1,\infty) &\mbox{if $d=2,$}\\
    & [1,3) &\mbox{if $d=3.$}
\end{aligned}
        \right.
    \end{equation}
    Then
    \begin{equation}
        \mathbb{G}_1\in H^{s,p}{(\mathbb{R}^d)}\iff s<\frac{d}{p}-d+2, \ \ d=2,3.
    \end{equation}
\end{Lemma}
{
\begin{proof}
The operator $1-\Delta$ defined on $L^p(\mathbb{R}^d)$ satisfies the hypotheses stated in Section \ref{s.def fractional}, so we can write 
\begin{equation}
    (1-\Delta)^{s/2}f(x)=\frac{\sin{(\pi s/2)}}{\pi}\int_0^\infty\lambda^{s/2-1}(1-\Delta)(1+\lambda-\Delta)^{-1}f(x)d\lambda,
\end{equation}
for every $f\in L^p(\mathbb{R}^d)$.
In particular, we can write
\begin{equation}
    (1-\Delta)^{s/2}\mathbb{G}_1(x)=\frac{\sin{(\pi s/2)}}{\pi}\int_0^\infty\lambda^{s/2-1}\mathbb{G}_{1+\lambda}(x)d\lambda
\end{equation}
    and, using the rescaling inequality \eqref{eq.estpg3},
    \begin{equation}
        \|  (1-\Delta)^{s/2}\mathbb{G}_1\|_{L^p(\mathbb{R}^d)}\lesssim\int_0^\infty \lambda^{s/2-1} |1+\lambda|^{\frac{d}{2}-1-\frac{d}{2p}}d\lambda<\infty,
    \end{equation}
    because $0< s <\frac{d}{p}-d+2$.
\end{proof}
}

We conclude this subsection with a pointwise estimate o{f} the fractional derivative of the Green function, { extending Lemma 2.1 in \cite{MOS18} for negative $s$ and $d=2$}.

Choose $\Psi{\in C^\infty_0([0,\infty))}$ to be  a smooth {non-negative} function, such that
\begin{equation}
    \Psi(r) =
  \left\{  \begin{aligned}
       & 1, \ \ \mbox{if} \ \ r<1;\\
       & 0, \ \ \mbox{if} \ \ r>2.
    \end{aligned}\right.
\end{equation}

\begin{Lemma}\label{l.dsgr06}
{ Let $d=2,3$.} For any $s \in ({2-d},2)$ and any {cut-off} radial function $\Psi(|x|)$ that is $1$ near the origin we have 
    \begin{eqnarray}
    |(-\Delta)^{s/2}  \mathbb{G}_1(x)| \lesssim  \Psi(|x|)|x|^{-d+2-s} + h_s(x), 
\end{eqnarray}
where $h_s(x)\in H^1 \cap C^2$ decays as  $|x|^{-d-s}$ at infinity and it is {constant} near the origin. 
{For any $s \in (-2,2-d]$,  we have 
    \begin{eqnarray}
    |(-\Delta)^{s/2}  \mathbb{G}_1(x)| \lesssim   h_s(x), 
\end{eqnarray}
where $h_s(x)\in H^1 \cap C^2$ decays as  $|x|^{-d-s}$ at infinity and it is constant near the origin.}

\end{Lemma}

\begin{proof}
{ We use the representation 
\begin{equation}
    \begin{aligned}
       (-\Delta)^{s/2}  \mathbb{G}_1(x) = \mathcal{F}^{-1} \left[ (2\pi)^{-d/2}\frac{|\xi|^s}{1+|\xi|^2}\right] (x),
    \end{aligned}
\end{equation}
where $\mathcal{F}$ and $\mathcal{F}^{-1}$ denote the Fourier transform and the inverse Fourier transform respectively, defined in \eqref{eq.Fourier} and \eqref{eq.inverse Fourier}. 
For $x$ bounded we use the relations
\begin{equation}
    \begin{aligned}
          &  \mathcal{F}^{-1} \left[ \frac{(2\pi)^{-d/2}|\xi|^2}{1+|\xi|^2}\right] (x) = J_1(x) + J_2(x)+J_3(x),  \\
            &J_1(x) = \mathcal{F}^{-1} \left[ (2\pi)^{-d/2}\frac{(1+|\xi|^2)^{s/2}}{1+|\xi|^2}\right] (x) ,\\
            &J_2(x)= -\frac{s}{2}\mathcal{F}^{-1} \left[ (2\pi)^{-d/2}\frac{(1+|\xi|^2)^{s/2-1}}{1+|\xi|^2}\right] (x),\\
           &J_3(x)=\mathcal{F}^{-1} \left[(2\pi)^{-d/2} \frac{|\xi|^s-(1+|\xi|^2)^{s/2}+s/2(1+|\xi|^2)^{s/2-1}}{1+|\xi|^2}\right](x) .
    \end{aligned}
\end{equation}
From \eqref{eq.fractional G transform}, because $s<2$, we can compute explicitly the inverse Fourier transforms
\begin{equation}
    \begin{aligned}
        &J_1(x) = \mathcal{F}^{-1} \left[ \frac{(2\pi)^{-d/2}}{(1+|\xi|^2)^{1-s/2}}\right] (x) = C(d,2-s)K_{\frac{d-2+s}{2}}(|x|)|x|^{\frac{2-d-s}{2}},\\
            &J_2(x)= -\frac{s}{2}\mathcal{F}^{-1} \left[ \frac{(2\pi)^{-d/2}}{(1+|\xi|^2)^{2-s/2}}\right] (x) = -\frac{s}{2}C(d,4-s)K_{\frac{d-4+s}{2}}(|x|)|x|^{\frac{4-d-s}{2}},\\
    \end{aligned}
\end{equation}
From the asymptotic expansions of the modified Bessel \eqref{eq.taylor 0 K}, for $x$ near zero, we have
\begin{equation}
   |J_1(x)\Psi(|x|)|\lesssim\left\{ \begin{aligned}
        &1, &s<d-2,\\
        &|x|^{2-s-d}, &s\geq d-2,
    \end{aligned}\right.
\end{equation}
and
\begin{equation}
   |J_2(x)\Psi(|x|)|\lesssim\left\{ \begin{aligned}
        &1, &s<d-4,\\
        &|x|^{4-s-d}, &s\geq d-4,
    \end{aligned}\right.
\end{equation}
so
\begin{equation}
   |(J_1(x)+J_2(x))\Psi(|x|)|\lesssim\left\{ \begin{aligned}
        &1, &s<d-2,\\
        &|x|^{2-s-d}, &s\geq d-2.
    \end{aligned}\right.
\end{equation}
Meanwhile 
\begin{equation}
   \begin{aligned}
        |J_3(0)|&\leq \int_{\mathbb{R}^d}\frac{\big||\xi|^s-(1+|\xi|^2)^{s/2}+s/2(1+|\xi|^2)^{s/2-1}\big|}{1+|\xi|^2}d\xi\\
        &\lesssim 1 + \int_{|\xi|\geq1}\frac{|\xi|^{s-4}}{1+|\xi|^2}d\xi \lesssim 1,
   \end{aligned}
\end{equation}
because $s<6-d$.

For $x$ large we use the stationary phase method. We consider $2-d<s<2$, the other case is simpler.
\begin{equation*}
    (-\Delta)^{s/2}\mathbb{G}_1(x)=\int_{\mathbb{R}^d}e^{ix\xi}\frac{|\xi|^s}{1+|\xi|^2}d\xi= I_1(x)+I_2(x),
\end{equation*}
with
\begin{equation*}
    \begin{aligned}
        I_1(x)&= \int_{\mathbb{R}^d}e^{ix\xi}\Psi(|\xi||x|)\frac{|\xi|^s}{1+|\xi|^2}d\xi,\\
        I_2(x)&=\int_{\mathbb{R}^d}e^{ix\xi}\big(1-\Psi(|\xi||x|)\big)\frac{|\xi|^s}
        {1+|\xi|^2}d\xi,
    \end{aligned}
\end{equation*}
for some $\Psi\in C^\infty([0,\infty))$ smooth compactly supported function such that $\Psi(x)=1$ near $x=0$.

For the first integral, we have
\begin{equation*}
    |I_1(x)|\leq \int_{|\xi|\leq\frac{C}{|x|}} \frac{|\xi|^s}{1+|\xi|^2}d\xi\leq C\int_0^\frac{C}{|x|}\rho^{s-2+d-1}d\rho=C|x|^{2-s-d},
\end{equation*}
and the integral is finite because $s>2-d$.
For the second integral we integrate by parts by means of the operator $(x/|x|^{-2},\nabla_\xi)$ $K$ times, obtaining
\begin{equation*}
    |I_2(x)|\leq \frac{C}{|x|^K}\int_{|\xi|\geq\frac{C}{|x|}} |q(x,\xi)| d\xi,
\end{equation*}
where
\begin{equation*}
    q(x,\xi)=\sum_{|\alpha|=K}\partial_\xi^\alpha\left(\big(1-\Psi(|x||\xi|)\big)\frac{|\xi|^s}{1+|\xi|^2}\right).
\end{equation*}
Since
\begin{equation*}
    |q(x,\xi)|\leq\frac{C(1+|\xi|)^{s-2}}{|\xi|^K},
\end{equation*}
choosing $K>s+d-2$, we get 
\begin{equation*}
    |I_2(x)|\leq \frac{C}{|x|^K}\int_{\frac{C}{|x|}}^\infty\frac{(1+\rho)^{s-2}}{\rho^K} \rho^{d-1}d\rho\leq \frac{C}{|x|^{d+s-2}},
\end{equation*}
because $s<2$.
Integrating by parts by means of the operator
$$(1+|z|^2)^{-1}(1+\Delta_\xi),$$
the decay factor $(1+|x|^2)^{M}$ can be obtained, for every  integer $M\geq1$.
}
\end{proof}

\section{$L^p$ domain and resolvent of $\Delta_\alpha$}\label{s.Lp}
In this section we want to extend the operator $-\Delta_\alpha$ to the $L^p$ spaces,  modifying slightly the definition \eqref{eq.op_dom}, associated with the $L^2$ setting. In the $L^p$ case we choose
$p$ so that
\begin{equation}\label{eq,aspq4}
    p \in \left\{
    \begin{aligned}
        &(1,\infty), \ \ \mbox{if $d=2,$} \\
        & (3/2,3) , \ \ \mbox{if $d=3.$}
    \end{aligned} \right.
\end{equation}
These intervals are the widest as possible, because they are the intersections of intervals of $p$ where these two conditions hold:
\begin{equation}
    \begin{aligned}
        &\mathbb{G}_\omega\in L^p(\mathbb{R}^d), &{H^{2,p}(\mathbb{R}^d)}\hookrightarrow L^{\infty}(\mathbb{R}^d).
    \end{aligned}
\end{equation}
We consider $\omega$ that satisfies \eqref{eq.2d1}, i.e.
\begin{equation}\label{eq.2d436}
\begin{aligned}
& \omega \in \mathbb{C} \setminus \sigma(\Delta_\alpha)  \\
    &\sigma(\Delta_\alpha) =  (-\infty,0] \cup \{E_\alpha\} =
     (-\infty,0] \cup \{\omega;  \alpha + c(\omega) = 0\}
\end{aligned}
\end{equation}
and we define the corresponding domain of $\Delta_\alpha$ as follows
\begin{equation}\label{eq.dp4}
    \mathcal{D}_p(\Delta_\alpha)\;=\;\Big\{\phi\in L^p(\mathbb{R}^d)\,\Big|\,\phi=g+\frac{ g(0)}{\alpha + c(\omega)}\,\mathbb{G}_{\omega}\textrm{ with } g = g_{\alpha,\omega}\in H^{2,p}(\mathbb{R}^d)\Big\}
\end{equation}
that we will denote with $H^{2,p}_\alpha(\mathbb{R}^d).$ 
The action of $\omega-\Delta_\alpha$ is defined by
\begin{equation}\label{eq.phg31}
    (\omega-\Delta_\alpha) \phi = (\omega-\Delta)g.
\end{equation}

It is not difficult to prove the following.
\begin{Lemma}\label{l.propH2pa}
  We have the following properties of $-\Delta_\alpha$:
\begin{enumerate}
\item[a)] the domain
$$  {H^{2,p}_\alpha(\mathbb{R}^d):=\mathcal{D}_p(\Delta_\alpha) } $$
is dense in $L^p{(\mathbb{R}^d)};$

\item[b)] the operator
\begin{equation}
(-\Delta_\alpha) : \mathcal{D}_p(\Delta_\alpha) \to L^p{(\mathbb{R}^d)}
\end{equation}
is a closed operator

    \item[c)] the domain and the action of $-\Delta_\alpha$ are independent of the choice of
$\omega$ satisfying  \eqref{eq.2d436};

.
\end{enumerate}
\end{Lemma}

\begin{proof}

To check a) we define
\begin{equation}
    S_0(\mathbb{R}^d) = \left\{ f \in S(\mathbb{R}^d); \partial_x^\alpha f(0)=0, \ \ \forall \alpha \in \mathbb{N}^d \right\}
\end{equation}
and see that it satisfies $S_0(\mathbb{R}^d) \subset H^{2,p}_\alpha(\mathbb{R}^d) \subset L^p(\mathbb{R}^d) $
and $S_0(\mathbb{R}^d)$ is dense in $L^p{(\mathbb{R}^d)}.$

The property b) follows from the relation \eqref{eq.phg31}, the fact that $\omega-\Delta$ is closed operator  and the property $g \in H^{2,p}{(\mathbb{R}^d)} \to g(0)$ is continuous functional,  due to the Sobolev embedding.
Finally, to prove c), it is sufficient to modify slightly Lemmas \ref{l.wun0} and \ref{l.ind2} with $L^p$ instead of $H^{1-\varepsilon}{(\mathbb{R}^d)}$ and $H^{2,p}{(\mathbb{R}^d)}$ instead of $H^2{(\mathbb{R}^d)}$.
\end{proof}

Our next step is to define the resolvent.
Since $L^2 \cap L^p$ is dense in $L^p$ and $0$ is in the resolvent set (in $L^2$ sense) of $\omega-\Delta_\alpha$ we can define  $(\omega-\Delta_\alpha)^{-1}$ (initially on $L^2 \cap L^p)$ by the formula

\begin{equation}\label{eq.res_formula87}
(\omega-\Delta_\alpha)^{-1} f\;=\;(\omega-\Delta)^{-1}f + \frac{1}{\alpha+ c(\omega)} \langle f, \mathbb{G}_\omega \rangle \mathbb{G}_\omega
\end{equation}
and  then extend it by density to $L^p.$
In fact the spectrum of the operator is
\begin{equation}
    \sigma(\Delta_\alpha) = (-\infty,0] \cup  \{\omega \in \mathbb{C}; \alpha + c(\omega) =0 \}.
\end{equation} and $\omega-\Delta$ is a sectorial operator that satisfies the classical inequality
\begin{equation} \label{eq.resDelta}
    \left\| (\omega-\Delta)^{-1}f \right\|_{L^p{(\mathbb{R}^d)}} \leq \frac{C}{|\omega|} \|f \|_{L^p{(\mathbb{R}^d)}}.
\end{equation}
From Proposition \ref{lemA3}, the inequality \eqref{eq.estpg3} combined with H\"older inequality gives
\begin{equation}\label{eq.Holder}
    \langle f,\mathbb{G}_\omega\rangle\lesssim |\omega|^{\frac{d}{2p}-1}\|f\|_{L^p{(\mathbb{R}^d)}},
\end{equation}
so the operator
\begin{equation}\label{eq.remB7}
   \mathbb B_{\omega} : f \in L^p{(\mathbb{R}^d)} \mapsto \frac{1}{\alpha+ c(\omega)} \langle f, \mathbb{G}_{\omega} \rangle \mathbb{G}_{\omega}
\end{equation}
satisfies the estimate
\begin{equation}
    \left\| \mathbb B_{\omega}(f) \right\|_{L^p{(\mathbb{R}^d)}} \leq \frac{C}{|\omega-E_\alpha|} \|f \|_{L^p{(\mathbb{R}^d)}}.
\end{equation}

Further, $\omega-\Delta_\alpha$ is a sectorial operator, namely it is a closed operator with spectrum in $(0,\infty)$ and such that for $\omega$ satisfying \eqref{eq.2d436} there exists a constant $C>0$ 
so that
\begin{equation} \label{eq.res72m}
    \left\| (\omega-\Delta_\alpha)^{-1}f \right\|_{L^p{(\mathbb{R}^d)}} \leq \frac{C}{d(\omega,\sigma(\Delta_\alpha))} \|f \|_{L^p{(\mathbb{R}^d)}},
\end{equation}
{ where $d(\omega,\sigma(\Delta_\alpha))$ is the distance of $\omega$ from the spectrum $\sigma(\Delta_\alpha)$.}

Finally we can assert that
\begin{equation}
    \begin{aligned}
   &  (\omega-\Delta_\alpha)^{-1}:L^p{(\mathbb{R}^d)} \to  H^{2,p}_\alpha{(\mathbb{R}^d)} ,\\
   & (\omega-\Delta_\alpha) (\omega-\Delta_\alpha)^{-1}f = f , \ f \in L^p{(\mathbb{R}^d)}\\
    & (\omega-\Delta_\alpha)^{-1} (\omega-\Delta_\alpha)\phi = \phi , \ \phi  \in H^{2,p}_\alpha{(\mathbb{R}^d)},\\
   & H^{2,p}_\alpha{(\mathbb{R}^d)} = \mathrm{Ran}(\omega-\Delta_\alpha)^{-1} =
   \left\{\phi = (\omega-\Delta_\alpha)^{-1}f;\ f \in L^p{(\mathbb{R}^d)} \right\}.
    \end{aligned}
\end{equation}
From the above properties we conclude also
\begin{equation}
    \mathrm{Ran} \ (\omega - \Delta_\alpha) = L^p.
\end{equation}

\subsection{The space {$H^{s,p}_\alpha(\mathbb{R}^d)$}}\label{subsection Hspa}

Once we know that $\omega-\Delta_\alpha$ is a sectorial operator, that satisfies \eqref{eq.res72m}, we can use the fractional powers of the operator as stated in the Subsection \ref{s.def fractional} (in alternative way we refer to \cite{H81}, \cite{EN2000}).

One possible way to introduce the spaces ${H^{s,p}_\alpha(\mathbb{R}^d)},$  {with $s\in(0,2)$ and $p$ that satisfies \eqref{eq. condition p}} is to define {them} as the ranges of the operator
$(\omega-\Delta_\alpha)^{-s/2}$.
 
With the  definition \eqref{eq.fractionalKomatsu negative} we can extend  the fractional power
$(\omega-\Delta_\alpha)^{-s/2}$  ,on the $L^p$ setting, considering the smallest closure of the operator 
\begin{equation}
    (\omega-\Delta_\alpha)^{-s/2} = \frac{\sin (s\pi/2)}{\pi}\int_0^\infty t^{-s/2} (\omega+t-\Delta_\alpha)^{-1} dt
\end{equation}
initially defined on $L^2\cap L^p$.
Using the resolvent formula \eqref{eq.res_formula87}, we can write in a more explicit way
\begin{equation}\label{eq.e12pf55}
    (\omega-\Delta_\alpha)^{-s/2}f=(\omega-\Delta)^{-s/2}f+\frac{\sin (s\pi/2)}{\pi}\int_0^\infty t^{-s/2} \mathbb{B}_{\omega+t} dt,
\end{equation}
with $\mathbb{B}_{\omega+t}$ defined in \eqref{eq.remB7}.
We note that, from the definition of $\alpha+c(\omega+t)$ in \eqref{eq.defc}, we can write
\begin{equation}\label{eq.w bound}
    {|\alpha+c(\omega+t)|\gtrsim |\omega+t|^{\frac{d}{2}-1},}
\end{equation}
that, combined with  \eqref{eq.Holder} and \eqref{eq.estpg3}, provides the estimate
\begin{equation}
    \|(\omega-\Delta_\alpha)^{-s/2}f\|_{L^p}\lesssim\|f\|_{L^p}\left(1+\int_0^\infty t^{-s/2}|\omega+t|^{-1}dt\right)
\end{equation}
{with} the integral finite for every $s\in(0,2)$.

Hence, the  operator $(\omega-\Delta_\alpha)^{-s/2}$ is {well-defined} on $L^p{(\mathbb{R}^d)}$ and satisfies the properties
\begin{enumerate}
    \item
    \begin{equation}
         (\omega-\Delta_\alpha)^{-s/2} : L^p{(\mathbb{R}^d)} \to L^p{(\mathbb{R}^d)},  \ \mbox{ for $p$ satisfying \eqref{eq. condition p}};
    \end{equation}

    \item $ (\omega-\Delta_\alpha)^{-s/2}$ is injective (see  Proposition 5.30 in \cite{EN2000}) and closed;

    \item the image $ \mathrm{Ran}(\omega-\Delta_\alpha)^{-s/2}$ of $(\omega-\Delta_\alpha)^{-s/2}$ is closed in $L^p{(\mathbb{R}^d)}$ (see { Corollary 7.4 in \cite{Ko66}}).
\end{enumerate}

Then we can use Definition 5.31 in \cite{EN2000} and define { the spaces $H^{s,p}_\alpha(\mathbb{R}^d)$}.
\begin{Definition}
Because $(\omega-\Delta_\alpha)^{-s/2}$ is injective and closed, we define $H^{s,p}_\alpha(\mathbb{R}^d)$ as 
\begin{equation}\label{eq.defdp72}
    H^{s,p}_\alpha{(\mathbb{R}^d)}: = \mathrm{Ran}(\omega-\Delta_\alpha)^{-s/2}.
\end{equation}
Moreover, we can define the operator 
\begin{equation}
   (\omega-\Delta_\alpha)^{s/2} :H^{s,p}_\alpha {(\mathbb{R}^d)}\to L^p{(\mathbb{R}^d)}
\end{equation}
as the inverse of
\begin{equation}
    (\omega-\Delta_\alpha)^{-s/2} : L^p{(\mathbb{R}^d)}  \ \ \to H^{s,p}_\alpha{(\mathbb{R}^d)}.
\end{equation}
\end{Definition}
 Note that we have   the relation
\begin{equation}
    \overline{\left(  (\omega-\Delta_\alpha)^{s/2}\Bigg|_{H^{2,p}_\alpha{(\mathbb{R}^d)}}\right)} =
    (\omega-\Delta_\alpha)^{s/2}\Bigg|_{H^{s,p}_\alpha{(\mathbb{R}^d)}}.
\end{equation}

{ 

\section{Sobolev embedding}\label{s.Sobolev embedding}

Our first step is to  generalise  the  following Sobolev inequality of  Lemma 2.2 in \cite{GMS24}, stating that
for any $q \in (2,\infty)$ there is a constant $C=C(q)>0$ so that for any $\phi \in H^1_\alpha{(\mathbb{R}^2)}$ we have $\phi \in L^q{(\mathbb{R}^2)}$ and
   \begin{equation}\label{eq.Sobolev2}
       \|\phi\|_{L^q(\mathbb{R}^2)} \leq C \|\phi\|_{H^1_\alpha{(\mathbb{R}^2)}}.
   \end{equation}

The key point is to use 
the fractional calculus formula \eqref{eq.e12pf55} that implies
   \begin{equation} \label{eq.frdiv6}
\begin{aligned}
   & (\omega-\Delta_\alpha)^{-s/2} (\varphi)(x)  - (\omega-\Delta)^{-s/2}(\varphi)(x)  \\
    &  =\frac{\sin (s\pi/2)}{\pi}\int_0^\infty t^{-s/2}  \langle \varphi, \mathbb{G}_{\omega+t} \rangle \mathbb{G}_{\omega+t}(x) \frac{dt}{\alpha+c(\omega+t)}  \\
    & =\frac{\sin (s\pi/2)}{\pi}\int_0^\infty t^{-s/2}  \mathbb B_{\omega+t} (\varphi)dt,
\end{aligned} 
   \end{equation}
where the operator $\mathbb{B}_\omega$ is defined in \eqref{eq.remB7}, i.e.
\begin{equation}
   \mathbb B_{\omega+t} (\varphi) = \langle \varphi, \mathbb{G}_{\omega+t} \rangle \mathbb{G}_{\omega+t}(x)    \frac{1}{\alpha+c(\omega+t)}. 
\end{equation}

\begin{Lemma}\label{l.sem1} { Let $d=2,3$} and let
  $ 1< p < q < \infty$ with $p$ satisfying \eqref{eq. condition p}. If $s>0$ satisfies

\begin{equation}\label{eq.sobr0}
       \begin{aligned}
        & s < \frac{d}{p}-d+2
       \end{aligned}
   \end{equation}
   and
   \begin{equation}\label{eq.sobr1}
       \begin{aligned}
        & \frac{s}{d} = \frac{1}{p}-\frac{1}{q},
       \end{aligned}
   \end{equation}
   then  there exists  a constant $C=C(d,p,q)>0$ so that for 
   any  $\phi \in H^{s,p}_\alpha{(\mathbb{R}^d)}$ we have $\phi \in L^q{(\mathbb{R}^d)}$ and
   \begin{equation}\label{eq.pSobolev} 
       \|\phi\|_{L^q(\mathbb{R}^{d})} \leq C \|\phi\|_{H^{s,p}_\alpha(\mathbb{R}^d)}.
   \end{equation}
\end{Lemma}

{\begin{Remark}
If $d=2,$ then \eqref{eq.sobr1} implies \eqref{eq.sobr0}. Let $d=3.$
     Theorem \ref{t.2.2} states that $H^{s,p}_\alpha(\mathbb{R}^3)=H^{s,p}(\mathbb{R}^3)$, provided $s<\frac{3}{p}-1$. Hence,  \eqref{eq.pSobolev} coincides with the standard Sobolev embedding. 
     There is a simple counterexample showing for $d=3$ the estimate \eqref{eq.pSobolev}  is not true for $s >\frac{3}{p}-1.$ This follows from the decomposition established in Theorems \ref{t.23A1} and the fact that for $s>\frac{3}{p}-1$ and \eqref{eq.sobr1} for $d=3$ we get $q >3.$ Indeed, in this case the singular part  $\mathbb{G}_\omega$ can not be in $ L^q(\mathbb{R}^3)$ for  $q\geq3.$
\end{Remark}

\begin{Remark}
In applications, we can use $H^{s_1,p}_\alpha(\mathbb{R}^d) \subset H^{s_2,p}_\alpha(\mathbb{R}^d)$ with $2 >s_1>s_2>0.$ Sufficient conditions for Sobolev embedding are
   \begin{equation}\label{eq.sobr2}
       \begin{aligned}
        &d-2 <\frac{d}{q}  \\
        &\frac{d}{p}-\frac{d}{q} \leq s.
       \end{aligned}
   \end{equation}
\end{Remark}

}

\begin{proof} 
The definition \eqref{eq.defdp72} of the perturbed Sobolev space
shows that we have to prove the inequality
\begin{equation}
       \|(\omega-\Delta_\alpha)^{-s/2}\varphi\|_{L^q(\mathbb{R}^d)} \leq C \|\varphi\|_{L^p(\mathbb{R}^d)}.
   \end{equation}
Obviously, we can use the classical Sobolev embedding and see that this inequality shall be established if we prove
\begin{equation}\label{eq.wmi28}
       \|\left[(\omega-\Delta_\alpha)^{-s/2} - (\omega-\Delta)^{-s/2} \right]\varphi\|_{L^{q}(\mathbb{R}^d)} \leq C \|\varphi\|_{L^p(\mathbb{R}^d)}, 
   \end{equation}

This inequality can be deduced from 

\begin{equation}\label{eq.wmi35}
       \|\left[(\omega-\Delta_\alpha)^{-s/2} - (\omega-\Delta)^{-s/2} \right]\varphi\|_{L^{q,\infty}(\mathbb{R}^d)} \leq C \|\varphi\|_{L^p(\mathbb{R}^d)}, 
   \end{equation}
   and Marcinkiewicz interpolation argument.

Indeed, starting with the formula \eqref{eq.frdiv6}, we 
can complete the proof, if we establish the inequality
\begin{equation}\label{eq.wmi1}
       \left\|\int_0^\infty t^{-s/2}  \langle \varphi, \mathbb{G}_{\omega+t} \rangle \mathbb{G}_{\omega+t}(x) \frac{dt}{\alpha+c(\omega+t)}\right\|_{L^{q,\infty}(\mathbb{R}^d)} \leq C \|\varphi\|_{L^p(\mathbb{R}^d)}, 
   \end{equation}

We take $\omega = 2 + E_\alpha$ and fix it.
Then we need the following pointwise estimate
\begin{equation}
    \begin{aligned}
      &  \left| \frac{\sin (s\pi/2)}{\pi}\int_0^\infty t^{-s/2}  \langle \varphi, \mathbb{G}_{\omega+t} \rangle \mathbb{G}_{\omega+t}(x)  \frac{dt}{\alpha+c(\omega+t)}\right|  \\
 &\lesssim\int_0^\infty t^{-s/2}  (1+t)^{(d-2)/2-d/(2p^\prime)} \left| \mathbb{G}_{\omega+t}(x) \right| \frac{dt}{|\alpha+c(\omega+t)|} \|\varphi\|_{L^p{(\mathbb{R}^d)}} 
    \end{aligned}
\end{equation}
where we have used \eqref{eq.Holder}.

Now we can use also the rescaling property \eqref{eq.rescalingG}, i.e. 
\begin{equation} \label{eq.defG76}
\mathbb{G}_{\omega+t}(x) \leq  (1+t)^{(d-2)/2} \mathbb{G}_1(\sqrt{\omega+t} x) ,
\end{equation}
and \eqref{eq.w bound} to write
\begin{equation}
    \begin{aligned}
      &  \left| (\omega-\Delta_\alpha)^{-s/2} (\varphi)(x)  - (\omega-\Delta)^{-s/2}(\varphi)(x) \right|  \\
      &\lesssim\left| \frac{\sin (s\pi/2)}{\pi}\int_0^\infty t^{-s/2}  \langle \varphi, \mathbb{G}_{\omega+t} \rangle \mathbb{G}_{\omega+t}(x)  \frac{dt}{\alpha+c(\omega+t)}\right|  \\
 &\lesssim\int_0^\infty t^{-s/2}  (1+t)^{d/(2p) - 1} \left| \mathbb{G}_{1}( \sqrt{\omega+t}x) \right|dt \|\varphi\|_{L^p{(\mathbb{R}^d)}} 
    \end{aligned}
    \end{equation}
To this end we can apply Lemma \ref{l.ste44} { with $A=\frac{d}{2p}$} and deduce
\begin{equation}
    \begin{aligned}
      &  \left| (\omega-\Delta_\alpha)^{-s/2} (\varphi)(x)  - (\omega-\Delta)^{-s/2}(\varphi)(x) \right|  \\
 &\lesssim|x|^{s-d/p} \|\varphi\|_{L^p{(\mathbb{R}^d)}} = |x|^{-d/q} \|\varphi\|_{L^p{(\mathbb{R}^d)}}.
    \end{aligned}
    \end{equation}
{ If $d=3$, we note that the integral  in the hypothesis of Lemma \ref{l.ste44} is convergent if and only if $s<\frac{3}{p}-1$, because $\mathbb{G}_1(\sqrt{\sigma})\sim\sigma^{-1/2}$, when $\sigma\to0$.}
    
    This completes the proof.

\end{proof}

It is not difficult to check the following
\begin{Lemma}\label{l.sem16} { Let $d=2,3$.
   For any $p \in (2,\infty)$ satisfying \eqref{eq. condition p} and any $s\in (0,2),$ satisfying \eqref{eq.sobr0}, then 
   }
   the space $H^2_\alpha{(\mathbb{R}^d)}$ is dense in $H^{s,p}_\alpha{(\mathbb{R}^d)}.$
\end{Lemma}

\begin{proof}
{ Firstly, we have to verify the inclusion $H^2_\alpha(\mathbb{R}^d)\subseteq H^{s,p}_\alpha(\mathbb{R}^d),$ that is equivalent to $H^{2-s}_\alpha(\mathbb{R}^d)\subseteq L^p(\mathbb{R}^d)$, stated by the Sobolev embedding of Lemma \ref{l.sem1}. 
To prove the density, for every $f\in H^{s,p}_\alpha(\mathbb{R}^d)$ and for every $\varepsilon>0$ we need to find an $f_\varepsilon\in H^2_\alpha(\mathbb{R}^d)$ such that $\|f-f_\varepsilon\|_{H^{s,p}_\alpha(\mathbb{R}^d)}\leq\varepsilon.$
From the definition of $ H^{s,p}_\alpha(\mathbb{R}^d)$, there exists $\phi\in L^p(\mathbb{R}^d)$ such that 
$$f=(\omega-\Delta_\alpha)^{-s/2}\phi,\ \ \omega=2+E_\alpha.$$
There exists a sequence of $\phi_n\in L^2\cap L^p (\mathbb{R}^d)$, such that $\phi_n\to \phi $ in $L^p(\mathbb{R}^d)$ when $n\to\infty$.
For every $\delta>0$, we define the function $H^2_\alpha(\mathbb{R}^d)\ni f_{\delta,n}=(\omega-\Delta_\alpha)^{-s/2}\omega(\omega-\delta\Delta_\alpha)^{-1+s/2}\phi_n$. 
The triangular inequality provides for every $n$ and $\delta$

\begin{equation}
    \begin{aligned}
        \|f_{\delta,n}-f\|_{H^{s,p}_\alpha(\mathbb{R}^d)}&=\|\omega(\omega-\delta\Delta_\alpha)^{-1+s/2}\phi_n-\phi\|_{L^p(\mathbb{R}^d)}\\
        &\leq \|\omega(\omega-\delta\Delta_\alpha)^{-1+s/2}\phi_n-\phi_n\|_{L^p(\mathbb{R}^d)}+\|\phi_n-\phi\|_{L^p(\mathbb{R}^d)}.
    \end{aligned}
\end{equation}
So, after choosing $n_0$ such that $\|\phi_{n_0}-\phi\|_{L^p(\mathbb{R}^d)}\leq \varepsilon/2$ , we can consider $\delta_1\ll1$, such that
$$\|\omega(\omega-\delta_1\Delta_\alpha)^{-1+s/2}\phi_{n_0}-\phi_{n_0}\|_{L^p(\mathbb{R}^d)}\leq \varepsilon/2.$$
The proof is concluded, choosing $f_\varepsilon=f_{\delta_1,n_0}$}.
\end{proof}

\section{Characterization of $H^{s,p}_\alpha{(\mathbb{R}^d)}$}\label{s.characterization}
We divide this Section in five Lemmas, the first one is  treating the embedding
\begin{equation}\label{eq.incsp4}
   H^{s,p}{(\mathbb{R}^d)} \subseteq H^{s,p}_\alpha{(\mathbb{R}^d)}, 
\end{equation}
the second subsection treats the 
inclusion
\begin{equation}\label{eq.spg9}
    \mathbb{G}_\omega \in H^{s,p}_\alpha{(\mathbb{R}^d)},
\end{equation}
the third one gives representation of $H^{s,p}_\alpha{(\mathbb{R}^d)}$ in the
case $0<s<d/p -(d-2)$, the fourth one treats the case $ d/p -(d-2) < s < d/p$ and the final subsection treats the case $s>d/p.$ In all this Section, we will assume $\omega >E_\alpha$, {and $s\in(0,2)$}.

\subsection{Statement on inclusion \eqref{eq.incsp4} }

\begin{Lemma}\label{l.inclusion1}
    
    If $p$ satisfies \eqref{eq. condition p} and
    \begin{equation}
        sp < d, d=2,3,
    \end{equation}
    then we have
   { \begin{equation}\label{eq.inclusion1}
       \left\|(\omega-\Delta_\alpha)^{s/2}\phi\right\|_{L^p(\mathbb{R}^d)} \lesssim \|(\omega-\Delta)^{s/2}\phi\|_{L^p(\mathbb{R}^d)},
    \end{equation}
    for every $\phi\in H^{s,p}(\mathbb{R}^d)$.}
\end{Lemma}

\begin{proof}
  
To show this estimate we use the fractional 
formula for positive powers \eqref{eq.fractionalKomatsu positive}

\begin{equation}
    (\omega-\Delta_\alpha)^{\frac{s}{2}}\phi=\frac{\sin (s\pi/2 )}{\pi} \int_0^\infty t^{\frac{s}{2}-1} (\omega-\Delta_\alpha)(\omega+t-\Delta_\alpha)^{-1} \phi dt
\end{equation}

To this end we use the resolvent identity \eqref{eq.res_formula} and we can write
\begin{equation}\label{eq:res_for7}
\begin{aligned}
    & (-\Delta_\alpha+\omega+t)^{-1} \phi\;=\;(-\Delta+\omega+t)^{-1}\phi + \frac{1}{\alpha + c(\omega+t) }\mathbb{G}_{\omega+t} \langle \phi, \mathbb{G}_{\omega+t} \rangle , \\
  &(-\Delta_\alpha+\omega+t)^{-1} (\omega-\Delta_\alpha) \phi\;=\;(-\Delta+\omega+t)^{-1}(\omega-\Delta)\phi - \frac{t}{\alpha+c(\omega+t) }\mathbb{G}_{\omega+t} \langle \phi, \mathbb{G}_{\omega+t} \rangle
\end{aligned}
\end{equation}

Hence we have
\begin{equation}\label{eq.defAfr45}
\begin{aligned}
  &  (\omega-\Delta_\alpha)^{s/2} \phi = (\omega-\Delta)^{s/2} \phi+ \frac{\sin (\pi s/2)}{\pi} \int_0^\infty t^{s/2}  \frac{1}{\alpha + c(\omega+t) }\mathbb{G}_{\omega+t} \langle \phi, \mathbb{G}_{\omega+t} \rangle  dt.
\end{aligned}
\end{equation}

Now our purpose is to check the estimate
\begin{equation}\label{mlpe50}
    \left\|\int_0^\infty t^{s/2}  \frac{1}{\alpha+c(\omega+t) }\mathbb{G}_{\omega+t} \langle \phi, \mathbb{G}_{\omega+t} \rangle  dt \right\|_{L^p{(\mathbb{R}^d)}} \lesssim \|\phi\|_{H^{s,p}{(\mathbb{R}^d)}}
\end{equation}
or equivalently
\begin{equation}\label{mlpe2}
    \left\|\int_0^\infty t^{s/2}  \frac{1}{\alpha + c(\omega+t) }\mathbb{G}_{\omega+t} \langle f, (\omega-\Delta)^{-s/2}\mathbb{G}_{\omega+t} \rangle  dt \right\|_{L^p{(\mathbb{R}^d)}} \lesssim \|f\|_{L^p{(\mathbb{R}^d)}},
\end{equation}
{ with $f=(\omega-\Delta)^{s/2}\phi$.}

Indeed, if \eqref{mlpe2} is verified, then from \eqref{eq.defAfr45} we get
\begin{equation}
    \|(\omega-\Delta_\alpha)^{s/2}\phi \|_{L^p{(\mathbb{R}^d)}} \lesssim \|(\omega-\Delta)^{s/2}\phi \|_{L^p{(\mathbb{R}^d)}} , 
\end{equation}
 for every $\phi \in H^{s,p},$ that can be read as
\begin{equation}\label{eq.ins2}
        \|\phi \|_{H^{s,p}_\alpha{(\mathbb{R}^d)}} \lesssim \|\phi \|_{H^{s,p}{(\mathbb{R}^d)}} .
    \end{equation}

So it remains to check \eqref{mlpe2} for $sp < d, d=2,3.$  We shall prove a weaker version
\begin{equation}\label{mlpe70}
    \left\|\int_0^\infty t^{s/2}  \frac{1}{\alpha + c(\omega+t) }\mathbb{G}_{\omega+t} \langle f, (\omega-\Delta)^{-s/2}\mathbb{G}_{\omega+t} \rangle  dt \right\|_{L^{p,\infty}{(\mathbb{R}^d)}} \lesssim \|f\|_{L^p{(\mathbb{R}^d)}}, \ sp \leq d
\end{equation}
and then via Marcinkiewicz interpolation theorem we deduce \eqref{mlpe2}.

For the purpose we use the Sobolev embedding
\begin{equation}
    \|(\omega-\Delta)^{-s/2} \mathbb G_{\omega+t}\|_{L^{p^\prime}{(\mathbb{R}^d)}} \lesssim \|\mathbb G_{\omega+t}\|_{L^q{(\mathbb{R}^d)}}, \ \  \frac{s}{d} = \frac{1}{q} -\frac{1}{p^\prime}
\end{equation}
where
$$ \frac{1}{q} = \frac{s+d}{d}-\frac{1}{p} $$
has to satisfy assumptions of {Proposition} \ref{lemA3}. For $d=2$ this assumption is $1/q \in (0,1]$ and it is satisfied {if} $sp \leq 2.$ For $d=3$ the requirement of 
{Proposition} \ref{lemA3} is $1/q \in (1/3,1]$ and it is satisfied if $sp \leq  3.$ Applying {Proposition} \ref{lemA3}, we deduce
\begin{equation}\label{eq.espp1}
    \|(\omega-\Delta)^{-s/2} \mathbb G_{\omega+t}\|_{L^{p^\prime}{(\mathbb{R}^d)}} \lesssim |\omega+t|^{(d-2)/2 - d/(2q)} = |\omega+t|^{-1-s/2+d/(2p)}
\end{equation}

Turning back to \eqref{mlpe2}, we apply \eqref{eq.espp1} in combination with the relation \eqref{eq.rescalingG} that gives
\begin{equation} \label{eq.defG92}
\mathbb{G}_{\omega+t}(x) = |\omega+t|^{(d-2)/2} \mathbb{G}_1(\sqrt{\omega+t} x) 
\end{equation}
so via 
{Proposition} \ref{lemA3} and { the change of variable $\sigma=\sqrt{\omega+t}|x|$} we get
\begin{equation}\label{mlpe93}
\begin{aligned}
   & \left|\int_0^\infty t^{s/2}  \frac{1}{\alpha + c(\omega+t) }\mathbb{G}_{\omega+t}(x) \langle f, (\omega-\Delta)^{-s/2}\mathbb{G}_{\omega+t} \rangle  dt \right|  \\ &\lesssim\left\| f \right\|_{L^p{(\mathbb{R}^d)}}\int_0^\infty t^{s/2}  \frac{1}{|\alpha + c(\omega+t)| } |\omega+t|^{(d-2)/2} |\omega+t|^{-1-s/2+d/(2p)} \left|\mathbb{G}_1(\sqrt{\omega+t} x)\right| dt \\
   &\lesssim \left\| f \right\|_{L^p{(\mathbb{R}^d)}}\int_0^\infty  |\omega+t|^{-1+d/(2p)} \left|\mathbb{G}_1(\sqrt{\omega+t} x)\right| dt \lesssim |x|^{-d/p}\left\| f \right\|_{L^p{(\mathbb{R}^d)}}.
\end{aligned}
\end{equation}
In this way the inequality \eqref{mlpe70} is established and the proof is complete.
  
\end{proof}

\subsection{Statement on inclusion \eqref{eq.spg9} }

\begin{Lemma}\label{l.GinH}
Let  $s \in {(}0, d/p),$ $d=2,3$ and let $ {\omega}\in \mathbb{C} \setminus \sigma(\Delta_\alpha).$ Then $\mathbb{G}_\omega \in H^{s,p}_\alpha{(\mathbb{R}^d)}.$
\end{Lemma}

{\begin{proof}
    We need to prove $\|(\omega-\Delta_\alpha)^{s/2}\mathbb{G}_\omega\|_{L^p(\mathbb{R}^d)}<\infty$, if $0<s<\frac{d}{p}$.
From \eqref{eq.fractionalKomatsu positive} we write
\begin{equation}\label{eq.fractional G 0}
   \begin{aligned}
        (\omega-\Delta_\alpha)^{s/2}\mathbb{G}_\omega=\frac{\sin{(\pi s/2)}}{\pi}\int_0^\infty t^{s/2-1} (t+\omega-\Delta_\alpha)^{-1}(\omega-\Delta_\alpha)\mathbb{G}_\omega  dt.
   \end{aligned}
\end{equation}
    So we start simplifying the expression $ (t+\omega-\Delta_\alpha)^{-1}(\omega-\Delta_\alpha)\mathbb{G}_\omega$.

Computing as in \eqref{eq:res_for7}, we have 
\begin{equation}\label{eq.fractionalG 1}
  \begin{aligned}
        (t+\omega-\Delta_\alpha)^{-1}(\omega-\Delta_\alpha)\mathbb{G}_\omega&=(t+\omega-\Delta)^{-1}(\omega-\Delta)\mathbb{G}_\omega- \frac{t\langle\mathbb{G_\omega},\mathbb{G}_{\omega+t}\rangle}{\alpha+c(\omega+t)}\mathbb{G}_{\omega+t}\\
        &=\left(1-\frac{t\langle\mathbb{G_\omega},\mathbb{G}_{\omega+t}\rangle}{\alpha+c(\omega+t)}\right)\mathbb{G}_{\omega+t}.
  \end{aligned}
\end{equation}
Thanks to the definition of the Green function \eqref{eq.HelmholtzDelta}
 and the property \eqref{eq.diff2G}, we can compute explicitly the scalar product 
 \begin{equation}
     \langle\mathbb{G_\omega},t \mathbb{G}_{\omega+t}\rangle=\langle\mathbb{G_\omega},(\omega-\Delta)(\mathbb{G}_\omega-\mathbb{G}_{\omega+t})\rangle=\langle\delta_0,\mathbb{G}_\omega-\mathbb{G}_{\omega+t}\rangle=(\mathbb{G}_\omega-\mathbb{G}_{\omega+t})(0)=-c(\omega)+c(\omega+t).
 \end{equation}
In this way \eqref{eq.fractionalG 1} becomes 
\begin{equation}
    (t+\omega-\Delta_\alpha)^{-1}(\omega-\Delta_\alpha)\mathbb{G}_\omega=\frac{\alpha+c(\omega)}{\alpha+c(\omega+t)}\mathbb{G}_{\omega+t}
\end{equation}
  and the fractional formula \eqref{eq.fractional G 0} for $\mathbb{G}_\omega$ can be written as
  \begin{equation}
      (\omega-\Delta_\alpha)^{s/2}\mathbb{G}_\omega=\frac{\sin{(s\pi/2)}}{\pi}\int_0^\infty t^{\frac{s}{2}-1}\frac{\alpha+c(\omega)}{\alpha+c(\omega+t)}\mathbb{G}_{\omega+t}dt.
  \end{equation}
The key part of this computation is the simplification of $c(\omega+t)$
 in the numerator, which, unlike the classical case, allows for a wider range of exponents for the convergence of the integral.
Indeed, with the estimates \eqref{eq.estpg3} and \eqref{eq.w bound} it is immediate to see 
\begin{equation}
    \|(\omega-\Delta_\alpha)^{s/2}\mathbb{G}_\omega\|_{L^p(\mathbb{R}^d)}\lesssim\int_0^\infty t^{\frac{s}{2}-1}\frac{\|\mathbb{G}_{\omega+t}\|_{L^p(\mathbb{R}^d)}}{|\alpha+c(\omega+t)|}dt\lesssim\int_0^\infty t^{\frac{s}{2}-1}|\omega+t|^{-\frac{d}{2p}}dt<\infty
\end{equation}
that is convergent in zero for $s>0$ and at infinity for $s<\frac{d}{p}$.
  
\end{proof}}

\subsection{The case $s<d/p -(d-2)$}
{ In this subsection we prove the inverse inclusion 
$$H^{s,p}_\alpha{(\mathbb{R}^d)}\subseteq H^{s,p}{(\mathbb{R}^d)}$$
It sufficient to show that the following inequality holds
\begin{equation}\label{eq.inclusion2}
      \| (\omega-\Delta)^{\frac{s}{2}} (\omega-\Delta_\alpha)^{-\frac{s}{2}}f\|_{L^p{(\mathbb{R}^d)}}\lesssim\|f\|_{L^p{(\mathbb{R}^d)}},
\end{equation}
when $s<\frac{d}{p}-d+2.$
Thanks to the equivalence
\begin{equation}
    \| (\omega-\Delta)^{\frac{s}{2}} F\|_{L^p{(\mathbb{R}^d)}} \sim \|  F\|_{L^p{(\mathbb{R}^d)}}  + \| (-\Delta)^{\frac{s}{2}} F\|_{L^p{(\mathbb{R}^d)}} 
\end{equation}
 we only need to prove
    \begin{equation}
      \left\| (-\Delta)^{\frac{s}{2}} \left(  (\omega-\Delta_\alpha)^{-\frac{s}{2}}f -  (\omega-\Delta)^{-\frac{s}{2}}f \right)\right\|_{L^p{(\mathbb{R}^d)}}\lesssim\|f\|_{L^p{(\mathbb{R}^d)}}.
\end{equation}
This estimate will be proven in the next Proposition with the additional hypothesis $s>\frac{d}{p}-d$, that in this case is always satisfied because $s>0$, but it will be crucial in the next subsection.

 \begin{Proposition}\label{pr.FractionalEstimate} Let $\omega>E_\alpha$. If $\frac{d}{p}-d<s<\frac{d}{p}-d+2$, then the following estimate holds
    \begin{equation}\label{eq.estp36}
      \left\| (-\Delta)^{\frac{s}{2}}\int_0^\infty t^{-s/2} \mathbb{B}_{\omega+t} (f)dt\right\|_{L^p{(\mathbb{R}^d)}}\lesssim\|f\|_{L^p{(\mathbb{R}^d)}},
\end{equation}
with $\mathbb{B}_\omega$ defined in \eqref{eq.remB7}.
\end{Proposition}
\begin{proof}
A direct computation with the rescaling inequality
\begin{equation}
    \|(-\Delta)^{\frac{s}{2}} \mathbb{G}_{\omega+t}\|_{L^p{(\mathbb{R}^d)}}\leq { |\omega+t|^{\frac{d}{2}-1-\frac{d}{2p}-\frac{s}{2}}\|(-\Delta)^{\frac{s}{2}}\mathbb{G}_1\|_{L^p{(\mathbb{R}^d)}}}
\end{equation}
and H\"older, is not sufficient, so we need again to estimate the integral pointwise, to apply Marcinkiewicz theorem.

We have
\begin{equation}\label{eq.1pt1}
\begin{aligned}
    & \left| (-\Delta)^{\frac{s}{2}}\int_0^\infty t^{-s/2} \mathbb{G}_{\omega+t}(x) \langle f,\mathbb{G}_{\omega+t} \rangle  \frac{dt}{\alpha + c(\omega+t)} \right|  \\
     & \lesssim\|f\|_{L^p{(\mathbb{R}^d)}}\int_0^\infty t^{-s/2} \left|  (-\Delta)^{\frac{s}{2}} \mathbb{G}_{\omega+t}(x) \right|
      \|\mathbb{G}_{\omega+t}\|_{L^{p^\prime}{(\mathbb{R}^d)}}\frac{dt}{|\omega+t|^{(d-2)/2}}  \\
      & \lesssim \|f\|_{L^p{(\mathbb{R}^d)}}\int_0^\infty t^{-\frac{s}{2}}|\omega+t|^{ -\frac{d}{2}+\frac{d}{2p}} \left|  (-\Delta)^{\frac{s}{2}} \mathbb{G}_{\omega+t}(x) \right| dt.
\end{aligned}
\end{equation}

Further we use the asymptotic
\begin{equation}
 | \partial_y^\alpha \mathbb{G}_\omega(y)| \lesssim \partial_r^k \mathbb G_0(r)  , \ \ r=|y| \lesssim 1, \ |\alpha|=k \leq 2,
\end{equation}
where
\begin{equation}
   \mathbb G_0(r) = 
   \left\{ \begin{aligned}
       & -\frac{1}{2\pi} \ln (r),  & d=2, \\
      & \frac{1}{4\pi r}, & d=3.  
    \end{aligned}\right.
    \end{equation}
For $r=|y| \gg 1$ we have
\begin{equation}
 | \partial_y^\alpha \mathbb{G}_\omega(y)| \lesssim  e^{-r}  , \ \ r=|y| \gg 1, \ |\alpha|=k \leq 2,
\end{equation}

Now we take cut off radial function $\Psi{\in C^\infty_0([0,\infty))}$ that is $1$ near the origin and we have
\begin{equation}
   \mathbb{G}_\omega(y) = \Psi(|y|)\mathbb G_0(|y|) + h(|y|), 
\end{equation}
where $h(y)$ is a $H^1 \cap C^1$ function that is $0$ near the origin and decays exponentially with all derivatives up to order $1.$
Then from Lemma \ref{l.dsgr06}, { if $s>2-d$} (see also \cite{BCS2015} for more details)
\begin{equation}\label{mest82}
    |(-\Delta)^{s/2}  {\mathbb{G}_1}(y)| \lesssim  \Psi(|y|)|y|^{-d+2-s} + h_s(y), 
\end{equation}
where $h_s(y)$ decays as  $|y|^{-d-s}$ at infinity and it is {constant } near the origin. { We note that $\frac{d}{p}-d<2-d$, but in the case $s\leq2-d$ the fractional derivative of $\mathbb{G}_1$ is not singular in zero, so the proof is easier.}

We use {\eqref{eq.rescalingG}}
\begin{equation}
    \mathbb{G}_{\omega+t}(x) = {|\omega+t|^{(d-2)/2} \mathbb{G}_1}( \sqrt{\omega+t} x),
\end{equation}
then \eqref{mest82} implies that
\begin{eqnarray}\label{mest97}
    |(-\Delta)^{s/2}  \mathbb{G}_{\omega+t}(x)| \lesssim  \Psi(\sqrt{|\omega+t|}|x|)|x|^{-d+2-s} + |\omega+t|^{s/2+(d-2)/2}h_s(\sqrt{\omega+t}x).
\end{eqnarray}
Now we can complete the estimates in \eqref{eq.1pt1} and find
\begin{equation}\label{eq.1pt07}
\begin{aligned}
     \left| (-\Delta)^{\frac{s}{2}}\int_0^\infty t^{-s/2} \mathbb{G}_{\omega+t}(x) \langle f,\mathbb{G}_{\omega+t} \rangle  \frac{dt}{\alpha + c(\omega+t)} \right| \lesssim \|f\|_{L^p{(\mathbb{R}^d)}}(I_1(x)+I_2(x)),
\end{aligned}
\end{equation}
with
\begin{equation}
    I_1(x)= \int_0^\infty t^{-\frac{s}{2}}|\omega+t|^{-\frac{d}{2}+\frac{d}{2p}} \Psi (\sqrt{\omega+t}|x|)|x|^{-d+2-s}   dt 
\end{equation}
and 
\begin{equation}
\begin{aligned}
     I_2(x)&= \int_0^\infty t^{-\frac{s}{2}} |\omega+t|^{-\frac{d}{2}+\frac{d}{2p}}|\omega+t|^{\frac{s}{2}+\frac{d}{2}-1} h_s(\sqrt{\omega+t} |x|)
      dt \\
      & = \int_0^\infty t^{-\frac{s}{2}} |\omega+t|^{\frac{s}{2}+\frac{d}{2p}-1} h_s(\sqrt{\omega+t} |x|)
      dt.
\end{aligned}
\end{equation}
To estimate $I_1$ we only need to compute explicitly the integral
\begin{equation}\label{eq.estimate I1}
    I_1(x)\lesssim\int_0^{\frac{2}{|x|^2}} t^{-\frac{s}{2}-\frac{d}{2}+\frac{d}{2p}} dt|x|^{-d+2-s}   \lesssim |x|^{-\frac{d}{p}}.
\end{equation}
We used that and the fact that $\frac{d}{2p}-\frac{d}{2}<0$. We note that the integral is not divergent in zero thanks to the hypothesis $s<\frac{d}{p}-d+2$. 

To estimate $I_2$ we distinguish two cases. We first consider the case ${|x|>\frac{1}{\sqrt{\omega}}}$, that implies $\sqrt{|\omega+t|}|x|>1$. So considering the behaviour of $h_s$ we have
\begin{equation}
    I_2(x)\mathbb{1}_{\{|x|>{\frac{1}{\sqrt{\omega}}}\}}\lesssim\int_0^\infty t^{-\frac{s}{2}}|\omega+t|^{\frac{d}{2p}-1-\frac{d}{2}}dt |x|^{-d-s}\mathbb{1}_{\{|x|>{\frac{1}{\sqrt{\omega}}}\}}.
\end{equation}
The integral is finite because $s<2$ and $s>\frac{d}{p}-d$, so
\begin{equation}\label{eq.estimate I2.1}
    I_2(x)\mathbb{1}_{\{|x|>{\frac{1}{\sqrt{\omega}}}\}}\lesssim |x|^{-d-s}\mathbb{1}_{\{|x|>{\frac{1}{\sqrt{\omega}}}\}}\lesssim |x|^{-\frac{d}{p}},
\end{equation}
also because $s>\frac{d}{p}-d$.
We turn to the case $|x|\leq{\frac{1}{\sqrt{\omega}}}$  and we split the integral in two intervals. Using the behaviors of $h_s$, we can write
\begin{equation}\label{eq.estimate I2.2}
\begin{aligned}
   I_2(x)\mathbb{1}_{\{|x|\leq{\frac{1}{\sqrt{\omega}}}\}}
   \lesssim&\int_0^{\frac{1}{|x|^2}-{\omega}}t^{-\frac{s}{2}-\frac{d}{2}+\frac{d}{2p}} dt|x|^{-s-d+2}{\mathbb{1}_{\{|x|\leq\frac{1}{\sqrt{\omega}}\}}}\\
   &+\int_{\frac{1}{|x|^2}-{\omega}}^\infty t^{-\frac{s}{2}}|\omega+t|^{-\frac{d}{2}+\frac{d}{2p}-1} dt|x|^{-s-d}\mathbb{1}_{\{|x|\leq{\frac{1}{\sqrt{\omega}}}\}}\\
    \lesssim&|x|^{-\frac{d}{p}}{\mathbb{1}_{\{|x|\leq\frac{1}{\sqrt{\omega}}\}}}+ |x|^{-s-d}{(|x|^{-2}-\omega)^{-\frac{s}{2}-\frac{d}{2}+\frac{d}{2p}}}\mathbb{1}_{\{|x|\leq{\frac{1}{\sqrt{\omega}}}\}}\lesssim|x|^{-\frac{d}{p}},
\end{aligned}
\end{equation}
{
where, in  the last inequality, we used $(|x|^{-2}-\omega)\mathbb{1}_{\{|x|\leq{\frac{1}{\sqrt{\omega}}}\}}\gtrsim|x|^{-2}\mathbb{1}_{\{|x|\leq{\frac{1}{\sqrt{\omega}}}\}}$ and the hypothesis $s>\frac{d}{p}-d$.

We give more details how to estimate the two integrals 
 in \eqref{eq.estimate I2.2}. In the  interval $0\leq t\leq|x|^{-2}-\omega$ we have $\sqrt{\omega+t}|x|<1$, so $h_s(\sqrt{\omega+t}|x|)\lesssim1$. So we can write 
\begin{equation}
    \begin{aligned}
        &\int_0^{\frac{1}{|x|^2}-\omega} t^{-\frac{s}{2}} |\omega+t|^{\frac{s}{2}+\frac{d}{2p}-1} dt\mathbb{1}_{\{|x|\leq\frac{1}{\sqrt{\omega}}\}}= \int_0^{\frac{1}{|x|^2}-\omega} t^{-\frac{s}{2}} |\omega+t|^{-\frac{d}{2}+\frac{d}{2p}+\frac{s}{2}+\frac{d}{2}-1} dt \mathbb{1}_{\{|x|\leq\frac{1}{\sqrt{\omega}}\}}\\
        &\lesssim  \int_0^{\frac{1}{|x|^2}-\omega} t^{-\frac{s}{2}-\frac{d}{2}+\frac{d}{2p}}dt |x|^{-s-d+2}\mathbb{1}_{\{|x|\leq\frac{1}{\sqrt{\omega}}\}}\lesssim |x|^{-\frac{d}{p}}\mathbb{1}_{\{|x|\leq\frac{1}{\sqrt{\omega}}\}}.
    \end{aligned}
\end{equation}
In order we used: \begin{itemize}
    \item $|\omega+t|^{-\frac{d}{2}+\frac{d}{2p}}\leq t^{-\frac{d}{2}+\frac{d}{2p}}$, given by $-\frac{d}{2}-\frac{d}{2p}<0$;\\
    \item $|\omega+t|^{\frac{s}{2}+\frac{d}{2}-1}< |x|^{-s-d+2}$, given by $|\omega+t|<\frac{1}{|x|^2}$ and $s+d-2\geq0$;
    
     \item$\int_0^{\frac{1}{|x|^2}-\omega} t^{-\frac{s}{2}-\frac{d}{2}+\frac{d}{2p}}dt=\frac{1}{1-\frac{s}{2}-\frac{d}{2}+\frac{d}{2p}}(|x|^{-2}-\omega)^{1-\frac{s}{2}-\frac{d}{2}+\frac{d}{2p}}\lesssim |x|^{s+d-\frac{d}{p}-2},$ where the convergence of the integral and the inequality are given by the hypothesis $s<\frac{d}{p}-d+2$.
\end{itemize}
For the interval $t\geq |x|^{-2}-\omega$, we use the fact that $h_s(\sqrt{\omega+t}|x|)\lesssim |\omega+t|^{-\frac{d}{2}-\frac{s}{2}}|x|^{-d-s}$ and the hypothesis $s>\frac{d}{p}-p$, and we have
\begin{equation}
\begin{aligned}
    \int_{\frac{1}{|x|^2}-\omega}^\infty t^{-\frac{s}{2}}|\omega+t|^{\frac{s}{2}+\frac{d}{2p}-1}h_s(\sqrt{\omega+t}|x|)dt&\lesssim\int_{\frac{1}{|x|^2}-\omega}^\infty t^{-\frac{s}{2}}|\omega+t|^{-\frac{d}{2}+\frac{d}{2p}-1}dt|x|^{-d-s}\\
    &\lesssim |x|^{-s-d}(|x|^{-2}-\omega)^{-\frac{s}{2}-\frac{d}{2}+\frac{d}{2p}}.
    \end{aligned}
\end{equation}

}

{Finally, substituting \eqref{eq.estimate I1}, \eqref{eq.estimate I2.1} and \eqref{eq.estimate I2.2} in \eqref{eq.1pt07}}, we arrive at
\begin{equation}\label{eq.estp09}
      \left| (-\Delta)^{\frac{s}{2}}\int_0^\infty t^{-s/2} \mathbb B_{\omega+t} (f)(x)dt\right|\lesssim |x|^{-d/p}\|f\|_{L^p{(\mathbb{R}^d)}},
\end{equation}
that implies
\begin{equation}\label{eq.LpwDelta}
      \left\| (-\Delta)^{\frac{s}{2}}\int_0^\infty t^{-s/2} \mathbb B_{\omega+t} (f)dt\right\|_{L^{p,\infty}{(\mathbb{R}^d)}}\lesssim\|f\|_{L^p{(\mathbb{R}^d)}}.
\end{equation}

We  complete the proof of \eqref{eq.estp36} by using Marcinkiewicz  interpolation theorem.
\end{proof}
The two embeddings \eqref{eq.inclusion2}, valid for $s<\frac{d}{p}-d+2$, and {\eqref{eq.inclusion1}}, valid for $s<\frac{d}{p}$ prove the equivalence of the two spaces.
\begin{Lemma}
    If ${0<}s<\frac{d}{p}-d+2$, then the perturbed domain coincides with the unperturbed one:
    $$H^{s,p}_\alpha{(\mathbb{R}^d)}=H^{s,p}{(\mathbb{R}^d)}.$$
\end{Lemma}

\subsection{The case $s>d/p-(d-2)$}
 The main idea of this subsection is that the difference of two Green functions is more regular as stated in subsection \ref{s.regularityG}.
 
    We can decompose
    \begin{equation}\label{cdec5m}
         (\omega-\Delta_\alpha)^{-s/2}f = (\omega-\Delta)^{-s/2}f + \Gamma_s(f) +   C_s(f) \mathbb G_\omega  ,
      \end{equation}
      where
    \begin{equation}
       { \Gamma_s} : f \in L^p \mapsto \Gamma(f)
    \end{equation}
    is defined (at least for $f \in S(\mathbb{R}^d)$ by
   
    \begin{equation}\label{eq.defGamma}
        \begin{aligned}
           & \Gamma_s(f) =  -\frac{\sin\left(\frac{s\pi}{2}\right)}{\pi }\int_0^\infty   t^{-s/2}  \langle f, \mathbb{G}_{\omega+t} \rangle (\mathbb{G}_{\omega+t}-\mathbb{G}_\omega)\frac{dt}{\alpha + c(\omega+t)}  ,
        \end{aligned}
    \end{equation}

    and
    \begin{equation} \label{eq.defC}
     C_s:f \in S(\mathbb{R}^2) \mapsto  C_s(f) =   -\frac{\sin\left(\frac{s\pi}{2}\right)}{\pi }\int_0^\infty t^{-s/2}  \langle f, \mathbb{G}_{\omega+t} \rangle \frac{dt}{\alpha + c(\omega+t)}.
      \end{equation}

    \begin{Proposition}\label{pr.C}
        If ${\frac{d}{p}-d+2<s<2}$, we  have the estimate
        \begin{equation} \label{eq.escf60}
       | C_s(f) | \lesssim \|f\|_{L^p{(\mathbb{R}^d)}}.
    \end{equation}
    \end{Proposition}
\begin{proof}

We use the estimates { \eqref{eq.Holder}, \eqref{eq.w bound},} therefore
\begin{equation}\label{eq.escf82}
\begin{aligned}
        |C_s(f)| &\lesssim \int_0^\infty t^{-s/2} |\omega+t|^{\frac{d}{2p}-1} \frac{dt}{|\alpha+c(\omega+t))|} \|f\|_{L^p{(\mathbb{R}^d)}}  \\ &\lesssim\int_0^1{t^{-\frac{s}{2}}}{dt}\|f\|_{L^p{(\mathbb{R}^d)}}+\int_1^\infty {t^{-\frac{s}{2}+\frac{d}{2p}-\frac{d}{2}}}{dt}\|f\|_{L^p{(\mathbb{R}^d)}}\lesssim\|f\|_{L^p{(\mathbb{R}^d)}}.
\end{aligned}
\end{equation}
\end{proof}
\begin{Proposition}\label{pr.Gamma} {If $\frac{d}{p}-d+2<s<2$,  we have the estimate}
    $$\| \Gamma_s(f) \|_{H^{s,p}{(\mathbb{R}^d)}} \lesssim \|f\|_{L^p{(\mathbb{R}^d)}}.$$
\end{Proposition}

\begin{proof}
{The $L^p$ estimate is immediate, because }
 \begin{equation}\label{eq.esgf64bis}
 \begin{aligned}
     \| \Gamma_s(f) \|_{L^p{(\mathbb{R}^d)}} \lesssim  \int_0^\infty t^{-\frac{s}{2}}(\omega+t)^{\frac{d}{2p}-\frac{d}{2}}\|\mathbb{G}_{\omega+t}\|_{L^p{(\mathbb{R}^d)}}dt \|f\|_{L^p{(\mathbb{R}^d)}}+|C_s(f)|\lesssim\|f\|_{L^p{(\mathbb{R}^d)}}.
 \end{aligned}
    \end{equation}

 We have to estimate the integral
\begin{equation}
   (\omega-\Delta)^{\frac{s}{2}} \int_0^\infty   t^{-s/2}  \langle f, \mathbb{G}_{\omega+t} \rangle (\mathbb{G}_{\omega+t}-\mathbb{G}_\omega)\frac{dt}{\alpha + c(\omega+t)},
\end{equation}
that can be written as 
\begin{equation}
   (\omega-\Delta)^{\frac{s}{2}-1} \int_0^\infty   -t^{-s/2+1} \mathbb B_{\omega+t}(f){dt},
\end{equation}
because
$$(\omega-\Delta)^{\frac{s}{2}}(\mathbb{G}_{\omega+t}-\mathbb{G}_\omega)=(\omega-\Delta)^{\frac{s}{2}-1}(-t\mathbb{G}_{\omega+t}).$$
To estimate the fractional derivative, we note that the exponent $m=s-2<0$ satisfies the condition $\frac{d}{p}-d<m<\frac{d}{p}-d+2$, so  we can apply the Proposition \ref{pr.FractionalEstimate}.  The estimate \eqref{eq.estp36} gives
\begin{equation}
    \left\|(-\Delta)^{\frac{s}{2}-1} \int_0^\infty   -t^{-s/2+1}  \mathbb B_{\omega+t}(f) {dt}\right\|_{L^p{(\mathbb{R}^d)}}\lesssim \|f\|_{L^p{(\mathbb{R}^d)}},
\end{equation}
    that concludes the proof.
\end{proof}
Combining these results we obtain the decomposition  in a regular and singular part. 
\begin{Lemma} \label{l.decomp1} If  ${\frac{d}{p}-d+2<s<2}$ and $f\in L^p{(\mathbb{R}^d)}$, there exist unique  $g\in H^{s,p}{(\mathbb{R}^d)}$ and $C_s(f)$ functional, such that
\begin{equation}
    (\omega-\Delta_\alpha)^{-\frac{s}{2}}f=g+C_s(f)\mathbb{G}_\omega.
\end{equation}
\end{Lemma}
\begin{proof}
It is sufficient to apply Propositions \ref{pr.C} and \ref{pr.Gamma} with the elliptic estimate
$$ \| (\omega-\Delta)^{-\frac{s}{2}}f \|_{H^{s,p}{(\mathbb{R}^d)}} \lesssim \|f\|_{L^p{(\mathbb{R}^d)}}.$$
In fact, we can choose
$g=g(f)= (\omega-\Delta)^{-\frac{s}{2}}f+\Gamma_s(f)\in H^{s,p}$, with $\Gamma_s$ defined in \eqref{eq.defGamma}
and $C_s(f)$ the functional defined in \eqref{eq.defC}. The uniqueness follows from the fact that $\mathbb{G}_\omega\in H^{s,p}{(\mathbb{R}^d)}\iff s<\frac{d}{p}-d+2$.
\end{proof}
\subsection{The case $\frac{d}{p}-d+2<s<\frac{d}{p}$}
We underline again that this case is non trivial only in dimension $d=3$. We have all the ingredients to prove the following characterization
\begin{Lemma}
    If $\frac{d}{p}-d+2<s<\frac{d}{p}$, then
    $$H^{s,p}_\alpha{(\mathbb{R}^d)}=H^{s,p}{(\mathbb{R}^d)}+\text{Span}{\{\mathbb{G}_\omega}    \}.$$
\end{Lemma}
\begin{proof}
    Lemma \ref{l.GinH} says that $\mathbb{G}_\omega\in H^{s,p}_\alpha{(\mathbb{R}^d)}$, while Lemma \ref{l.inclusion1} gives the inclusion $H^{s,p}{(\mathbb{R}^d)}\subseteq H^{s,p}_\alpha{(\mathbb{R}^d)}$, proving 
    $$H^{s,p}{(\mathbb{R}^d)}+\text{Span}{\{\mathbb{G}_\omega}    \}\subseteq H^{s,p}_\alpha{(\mathbb{R}^d)}.$$
    The converse inclusion is a direct application of Lemma \ref{l.decomp1}.
\end{proof}
\subsection{The case $s>\frac{d}{p}$} We end this section with the final characterization, recovering the standard definition of $\mathcal{D}(-\Delta_\alpha)$.
\begin{Lemma}
    If ${\frac{d}{p}<s<2}$, then the following relation holds
    $$g(0)=(\alpha+c(\omega))C_s(f),$$
with $g$ and $C_s(f)$ as defined in Lemma \ref{l.decomp1} and $c(\omega)$ is defined in \eqref{eq.defc}.
In particular, for every $\phi\in H^{s,p}_{\alpha}$, there exists a unique $g\in H^{s,p}$ such that 
\begin{equation}
    \phi=g+\frac{g(0)}{\alpha+c(\omega)}\mathbb{G}_\omega.
\end{equation}
\end{Lemma}
\begin{proof}
    First of all we note that we are in the case $sp>d$, so the point value is well defined in $H^{s,p}( \mathbb{R}^d)$.
We have 
$$g(0)=(\omega-\Delta)^{-\frac{s}{2}}f (0)+\Gamma_s(f) (0),$$
with $\Gamma_s$ defined in \eqref{eq.defGamma}.
The asymptotic expansion of $\mathbb{G}_\omega$ gives
\begin{equation}\label{eq.evaluation01}
\begin{aligned}
    \Gamma_s(f)(0)&=-\frac{\sin\left(\frac{s\pi}{2}\right)}{\pi }\int_0^\infty t^{-\frac{s}{2}}\frac{\langle f,\mathbb{G}_{\omega+t}\rangle}{\alpha+c(\omega+t)}(-c(\omega+t)+c(\omega))dt\\
    &=\frac{\sin\left(\frac{s\pi}{2}\right)}{\pi }\int_0^\infty t^{-\frac{s}{2}}{\langle f,\mathbb{G}_{\omega+t}\rangle} dt +(\alpha+c(\omega))C_s(f),
    \end{aligned}
\end{equation}
    with $C_s(f)$ defined in \eqref{eq.defC}.
To compute the other term, we use the fractional resolvent formula
\begin{equation}\label{eq.evaluation02}
    \begin{aligned}
        (\omega-\Delta)^{-\frac{s}{2}}f (0)&=-\frac{\sin\left(\frac{s\pi}{2}\right)}{\pi }\int_0^\infty t^{-\frac{s}{2}}(t+\omega-\Delta)^{-1}f (0)dt\\
        &=-\frac{\sin\left(\frac{s\pi}{2}\right)}{\pi }\int_0^\infty t^{-\frac{s}{2}}\mathbb{G}_{\omega+t}*f (0)dt=-\frac{\sin\left(\frac{s\pi}{2}\right)}{\pi }\int_0^\infty t^{-\frac{s}{2}}{\langle f,\mathbb{G}_{\omega+t}\rangle} dt,
    \end{aligned}
\end{equation}
    where the last equality follows from the radiality of $\mathbb{G}_\omega$.

Adding \eqref{eq.evaluation01} and \eqref{eq.evaluation02} we have the thesis.
\end{proof}

\section{Local well posedness}\label{s.lwp}
In this section we consider the Cauchy problem \eqref{eq.CP81}
and we shall give alternative proof of the local existence result established in Theorem B.1 in  \cite{FGI22}.
{In the following, we will use the notation 
$$ L^p(0,T)V\coloneqq L^p((0,T);V(\mathbb{R}^d)) $$
and
$$ L^pV\coloneqq L^p(\mathbb R_+;V(\mathbb{R}^d)) $$
for $p\in[1,\infty]$, $s\ge 0$ and $V(\mathbb{R}^d)=L^p(\mathbb{R}^d),H^{s,p}_\alpha(\mathbb R^d)$ or $ H^{2,p}(\mathbb R^d)$ }
Further we recall the Strichartz estimates for $\Delta_\alpha$ that are obtained in \cite{CMY19,CMY19b} and \cite{DMSY18}.

\begin{equation} \label{eq.str710}
    \begin{aligned}
& \left\| e^{i t \Delta_\alpha} f \right\|_{L^q(0,T)L^r} \lesssim \|f\|_{L^2}, \\
& \left\|\int_0^t e^{i(t-\tau)\Delta_\alpha} F(\tau) d\tau\right\|_{L^q(0,T)L^r} \lesssim  \left\| F\right\|_{L^{\tilde{q}^\prime}(0,T)L^
{\tilde{r}^\prime}} 
    \end{aligned}
\end{equation}
provided  with $(r,q)$ and $(\tilde{r},\tilde{q})$ admissible pairs.
 The couple $(r,q)$ is admissible if
\begin{equation}\label{eq.adm18}
    \frac{2}{q} + \frac{d}{r} = \frac{d}{2}, \mbox{ with } \left\{\begin{aligned}
        &r\in[2,\infty),\ q\in(2,\infty], &d=2,\\
        &r\in[2,3),\ q\in (4,\infty], &d=3.
    \end{aligned}\right.
\end{equation}

We show how to obtain local {well-posedness} of the problem \eqref{eq.CP81}. 

\begin{proof}[Proof of Theorem \ref{t.le19}]
 Consider the operator
 \begin{equation}\label{eq.kop37}
     \mathfrak{K} (u) = e^{it \Delta_\alpha} u_0 -i{\mu} \int_0^t e^{i(t-\tau)\Delta_\alpha}  u(\tau)|u(\tau)|^{p-1} d\tau
 \end{equation}
    and {consider the Banach space} $$L^\infty(0,T)L^2 \cap L^{\tilde{q}^\prime}(0,T) L^{\tilde{r}^\prime}.$$ 

{ If $1\leq p<\frac{3}{2}$, we apply the Strichartz estimate \eqref{eq.str710} with

\begin{equation}\label{tqr46}
  \tilde{r} =2 ,\ \tilde{q}= \infty ,
\end{equation}
 and we get
  $$ \left\| \mathfrak{K}(u)\right\|_{L^q(0,T)L^r} \lesssim \|u_0\|_{L^2} + \left\| u|u|^{p-1}\right\|_{L^{1}(0,T)L^{2}  }  \lesssim \|u_0\|_{L^2} + \left\|u\right\|^p_{L^{p}(0,T)L^{2p}  } . $$
  This estimate holds for every $(q,r)$  admissible couple satisfying \eqref{eq.adm18}, so we can choose
  \begin{equation}
      \begin{aligned}
          & r=2p, \\
          & q = \frac{4p}{3(p-1)}.
      \end{aligned}
  \end{equation}
 Now  we need
 $$ \left\|u\right\|_{L^{p}(0,T)L^{2p}  }  \lesssim T^{\gamma}  \left\|u\right\|_{L^{q}(0,T)L^{2p}  }  $$
 with $\gamma=\frac{4-3(p-1)}{4p}>0$ and this can be done if 
 $$ p < 1+\frac{4}{3}=\frac{7}{3}, $$
that is fulfilled.
The estimate
 \begin{equation}\label{eq.sest60}
     \left\| \mathfrak{K}(u)\right\|_{L^{\frac{4p}{3(p-1)}}(0,T)L^{2p}} \lesssim \|u_0\|_{L^2} + T^{p\gamma} \left\| u\right\|^p_{L^{\frac{4p}{3(p-1)}}(0,T)L^{2p}}
 \end{equation}
    shows that $\mathfrak{K} $ maps
  \begin{equation}\label{eq.2R62}
     \left\{ u \in L^{\frac{4p}{3(p-1)}}(0,T)L^{2p}; \left\| u\right\|_{L^{\frac{4p}{3(p-1)}}(0,T)L^{2p}} \leq 2R \right\}
  \end{equation}
into itself,  provided $u_0 \in B_{L^2}(R)$
and $T=T(R,p)$ is sufficiently small.
In a similar way we deduce
$$ \left\| \mathfrak{K}(u)-\mathfrak{K}(\tilde{u})\right\|_{L^{\frac{4p}{3(p-1)}}(0,T)L^{2p}} \leq \frac{1}{2} \left\| u-\tilde{u}\right\|_{L^{\frac{4p}{3(p-1)}}(0,T)L^{2p}}  $$
so $\mathfrak{K}$ is a contraction in \eqref{eq.2R62}. Since we use contraction method we have that the solution is unique and depends continuously from the intial data.

If $\frac{4}{3}<p<2$, we consider $\varepsilon>0$ sufficiently small and the admissible couple 
\begin{equation}\label{tqr}
  \tilde{r} =\frac{3}{1+\varepsilon} ,\ \tilde{q}=\frac{4}{1-2\varepsilon}  ,
\end{equation}
with conjugate exponents
\begin{equation}
    \begin{aligned}
       & \tilde{r}^\prime =\frac{3}{2-\varepsilon} ,&\tilde{q}^\prime=\frac{4}{3+2\varepsilon}  .
    \end{aligned}
\end{equation}
We apply Strichartz estimate \eqref{eq.str710}
$$ \left\| \mathfrak{K}(u)\right\|_{L^q(0,T)L^r} \lesssim \|u_0\|_{L^2} + \left\| u|u|^{p-1}\right\|_{L^{\tilde{q}^\prime}(0,T)L^{\tilde{r}^\prime}  }  \lesssim \|u_0\|_{L^2} + \left\|u\right\|^p_{L^{\frac{4p}{3+2\varepsilon}}(0,T)L^{\frac{3p}{2-\varepsilon}}  }  $$
  and  we can choose the admissible couple
  \begin{equation}
      \begin{aligned}
          & r=\tilde{r}^\prime p=\frac{3p}{2-\varepsilon}, 
          & q =\frac{4p}{3p-4+2\varepsilon}.
      \end{aligned}
  \end{equation}
  We note that $\tilde{q}^\prime p<q$, and also that
  \begin{equation}
      \frac{1}{\tilde{q}^\prime p}=\frac{1}{q}+\frac{1}{p}
  \end{equation}so we can apply the H\"older inequality to get
 $$\|u\|^p_{L^{\tilde{q}\prime p}(0,T)L^{\tilde{r}^\prime p}}\leq T\|u\|^p_{L^{{q}}(0,T)L^{\tilde{r}^\prime p}},$$
 that provides
 \begin{equation}\label{eq.sest60 2}
     \left\| \mathfrak{K}(u)\right\|_{L^{\frac{4p}{3p-4+2\varepsilon}}(0,T)L^\frac{3p}{2-\varepsilon}} \lesssim \|u_0\|_{L^2} +T \left\|u\right\|^p_{L^{\frac{4p}{3p-4+2\varepsilon}}(0,T)L^{\frac{3p}{2-\varepsilon}}  } .
 \end{equation}
 Similarly, we get
\begin{equation}
     \left\| \mathfrak{K}(u)-\mathfrak{K}(\tilde{u})\right\|_{L^{\frac{4p}{3p-4+2\varepsilon}}(0,T)L^\frac{3p}{2-\varepsilon}}\leq \frac{1}{2}\|u-\tilde{u}\|_{L^{\frac{4p}{3p-4+2\varepsilon}}(0,T)L^\frac{3p}{2-\varepsilon}},
\end{equation}
so $\mathfrak{K}$ is a contraction.

 }

\end{proof}

Our Theorem \ref{t.2.1} guarantees the more general Strichartz estimates
\begin{equation} \label{eq.SH123}
    \begin{aligned}
& \left\| e^{i t \Delta_\alpha} f \right\|_{L^q(0,T)H^{s,r}_\alpha} \lesssim \|f\|_{H^s_\alpha}, \\
& \left\|\int_0^t e^{i(t-\tau)\Delta_\alpha} F(\tau) d\tau\right\|_{L^q(0,T)H^{s,r}_\alpha} \lesssim  \left\| F\right\|_{L^{\tilde{q}^\prime}(0,T)H^{s,{\tilde{r}^\prime}}_\alpha },
    \end{aligned}
\end{equation}
with $(r,q)$ and $(\tilde{r},\tilde{q})$ admissible pairs, that satisfy condition \eqref{eq.adm18}. Thanks to this estimates, we can extend the local existence result in the energy space $H^1_\alpha$. { We start from the dimension $d=2$, resuming the proof in \cite{GR25}, and improving it for $1<p<2$.}

\begin{proof}[Proof of Theorem \ref{t.Schrodinger2}]
    Consider the operator
  $$  \mathfrak{K} (u) = e^{it \Delta_\alpha} u_0 -i{\mu} \int_0^t e^{i(t-\tau)\Delta_\alpha}  u(\tau)|u(\tau)|^{p-1} d\tau. $$

Further, we {consider}  the Banach space {$\mathcal{B}=L^\infty(0,T)H^1_\alpha $} and the corresponding ball of radius $R$
  $$ B_{\mathcal{B}} = \left\{u \in \mathcal{B} ; \|u\|_{\mathcal{B}} \leq R \right\}.$$

Applying the Strichartz estimate \eqref{eq.SH123}, we find
   $$ \left\| \mathfrak{K}(u)\right\|_{L^{\infty}(0,T)H^1_\alpha} \lesssim \|u_0\|_{H^1_\alpha{(\mathbb{R}^2)}} + \left\| u|u|^{p-1}\right\|_{L^{\tilde{q}^\prime}(0,T)H^{1, \tilde{r}^\prime}_\alpha  } .   $$

Now we choose

  \begin{equation}\label{eq.rq}
  \begin{aligned}
      &\tilde{r} = \frac{2-\varepsilon}{1-\varepsilon}, &\tilde{q}= \frac{2(2-\varepsilon)}{\varepsilon}
  \end{aligned}
\end{equation}
  so that
  \begin{equation}\label{eq.rqprime}
  \begin{aligned}
      & \tilde{r}^\prime = 2-\varepsilon, &\tilde{q}^\prime= \frac{4-2\varepsilon }{4-3\varepsilon} .
  \end{aligned}
\end{equation}
and $\varepsilon\in (0,1)$ is sufficiently small.
Since $\tilde{r}^\prime <2,$ we see that  {Theorem \ref{t.2.2}} implies
$$ \left\| u|u|^{p-1}\right\|_{H^{1,{\tilde{r}^\prime}}_\alpha {(\mathbb{R}^2)} } \sim \left\| u|u|^{p-1}\right\|_{H^{1,{\tilde{r}^\prime}}{(\mathbb{R}^2)}  } \sim
\left\| u|u|^{p-1}\right\|_{L^{\tilde{r}^\prime}{(\mathbb{R}^2)}} + \left\| \nabla u|u|^{p-1}\right\|_{L^{\tilde{r}^\prime}{(\mathbb{R}^2)} }$$
Now we use the fact that $u \in H^1_\alpha{(\mathbb{R}^2)}$
$$  u= g + c_*\mathbb G_\omega$$
and we can continue the estimates  as follows

$$ \left\| |\nabla u||u|^{p-1}\right\|_{L^{\tilde{r}^\prime}{(\mathbb{R}^2)}} \leq \left\| |\nabla g||u|^{p-1}\right\|_{L^{\tilde{r}^\prime}{(\mathbb{R}^2)}} +
|c_*|\left\| |\nabla \mathbb{G}_\omega||u|^{p-1}\right\|_{L^{\tilde{r}^\prime}{(\mathbb{R}^2)}}.$$
Then we estimate each of the terms in the right side and find
\begin{equation}\label{eq.es33}
    \left\| |\nabla g||u|^{p-1}\right\|_{L^{\tilde{r}^\prime}{(\mathbb{R}^2)} }\lesssim\|\nabla  g\|_{L^2{(\mathbb{R}^2)}}\|u\|^{p-1}_{L^{\frac{2(2-\varepsilon)}{\varepsilon}(p-1)}{(\mathbb{R}^2)}}\lesssim \|\nabla  g\|_{L^2{(\mathbb{R}^2)}}\|u\|^{p-1}_{H^1_\alpha{(\mathbb{R}^2)}},
\end{equation}

\begin{equation} \label{eq.es37}
    \left\| |\nabla \mathbb{G}_\omega||u|^{p-1}\right\|_{L^{\tilde{r}^\prime} }\lesssim\|\nabla \mathbb{G}_\omega\|_{L^{2-\frac{\varepsilon}{2}}}\|u\|_{L^{\frac{(2-\varepsilon)(4-\varepsilon)}{\varepsilon}(p-1)}}^{p-1}\lesssim \|u\|^{p-1}_{H^1_\alpha}.
\end{equation}

We used the Sobolev embedding \eqref{eq.Sobolev2}, valid for $q\in(2,\infty)$. Hence
\begin{equation}
 \left\| |\nabla u||u|^{p-1}\right\|_{L^{\tilde{r}^\prime}{(\mathbb{R}^2)} } \lesssim    \left( \|\nabla  g\|_{L^2{(\mathbb{R}^2)}}+ |c_*| \right) \|u\|^{p-1}_{H^1_\alpha{(\mathbb{R}^2)}}
\end{equation}
and via the definition of $H^1_\alpha$ norm, we get
\begin{equation}\label{mhes2}
 \left\| |\nabla u||u|^{p-1}\right\|_{L^{\tilde{r}^\prime}{(\mathbb{R}^2)} } \lesssim     \|u\|^{p}_{H^1_\alpha{(\mathbb{R}^2)}}.
\end{equation}

So

$$ \left\| \mathfrak{K}(u)\right\|_{L^\infty(0,T)H^1_\alpha} \lesssim \|u_0\|_{H^1_\alpha{(\mathbb{R}^2)}} +  T^{1/\tilde{q}'}\|u\|^{p}_{L^\infty(0,T) ,H^1_\alpha}$$

For $1<p \leq 2$ we can follow the Kato argument in \cite{K95} and in Section 4.4 of \cite{Cazenave}. { In this way, we deduce the uniqueness and continuous dependence of the solution on the initial data.}

For $p >2$ we can verify the inequality
\begin{equation}\label{eq.contraction difference}
    \left\| \mathfrak{K}(u)-\mathfrak{K}(v)\right\|_{L^{\infty}(0,T) H^1_\alpha} \lesssim \frac{1}{2} \left\| u-v\right\|_{L^\infty(0,T) H^1_\alpha}  
\end{equation}
so $\mathfrak{K}$ is a contraction in $L^\infty(0,T)H^1_\alpha. $ { This ensures the uniqueness of the solution and its continuous dependence on the initial data.}
 To verify \eqref{eq.contraction difference}, we call with $f(u)=|u|^{p-1}u$ the non linearity and we see that the Strichartz estimate \eqref{eq.SH123} provides
\begin{equation}\label{eq.ku-kv}
    \|\mathfrak{K}(u)-\mathfrak{K}(v)\|_{L^{\infty}(0,T) H^1_\alpha} \leq C\left\| f(u)-f(v)\right\|_{L^{\tilde{q}^\prime}(0,T) H^{1,\tilde{r}^\prime}_\alpha}
\end{equation}
for any $(\tilde{q},\tilde{r})$ admissible couple, with a constant $C$ that is independent from $u$ and $v$.
We consider again $\tilde{r}^\prime=2-\varepsilon$, so we have the equivalence 
$$H^{1,\tilde{r}^\prime}_\alpha{(\mathbb{R}^2)}=H^{1,\tilde{r}^\prime}{(\mathbb{R}^2)}$$
from {Theorem \ref{t.2.2}}.

The estimate \eqref{eq.contraction difference}  follows from
\begin{equation}\label{eq.con dif7}
    \left\| f(u)-f(v)\right\|_{ H^{1,\tilde{r}^\prime}_\alpha{(\mathbb{R}^2)}} \lesssim  \left\| u-v\right\|_{ H^1_\alpha{(\mathbb{R}^2)}}  \left( \left\| u\right\|^{p-1}_{ H^1_\alpha{(\mathbb{R}^2)}} + \left\| v\right\|^{p-1}_{ H^1_\alpha{(\mathbb{R}^2)}} \right) 
\end{equation}
that holds for every $u$ and $v$ in $H^1_\alpha{(\mathbb{R}^2)}$.

Because $p>2$, we have the inequalities
\begin{equation}\label{eq.u-v0}
    \begin{aligned}
        &|f(u)-f(v)|\lesssim |u-v|\big(|u|^{p-1}+|v|^{p-1}\big)
    \end{aligned}
\end{equation}
and 

\begin{equation}\label{eq.u-v80}
    \begin{aligned}
        |\nabla(f(u)-f(v))|&\lesssim|\nabla(u-v)|\big(|u|^{p-1}+|v|^{p-1}\big)  \\
      & + |(u-v)|\big(|\nabla u||u|^{p-2}+|\nabla v||v|^{p-2}\big) .
    \end{aligned}
\end{equation}

To estimate the term $|\nabla(u-v)|\big(|u|^{p-1}+|v|^{p-1}\big) $ in $L^{\tilde{r}^\prime}$ we use the argument of the proof of \eqref{mhes2}
and find
\begin{equation}\label{mhes90}
 \left\| |\nabla (u-v)|\big(|u|^{p-1}+|v|^{p-1}\big)\right\|_{L^{\tilde{r}^\prime}{(\mathbb{R}^2)} } \lesssim _p    \|u-v\|_{H^1_\alpha{(\mathbb{R}^2)}} \big( \|u\|^{p-1}_{H^1_\alpha{(\mathbb{R}^2)}} + \|v\|^{p-1}_{H^1_\alpha{(\mathbb{R}^2)}}\big).
\end{equation}

For 
$ |(u-v)|\big(|\nabla u||u|^{p-2}+|\nabla v||v|^{p-2}\big)$  we follow the same idea as done in \eqref{eq.es33} , decomposing $\nabla u=\nabla g+c_*\nabla\mathbb{G}_\omega$, with $g\in H^1{(\mathbb{R}^2)}$.
\begin{equation}\label{eq.es01}
\begin{aligned}
    &\left\|(u-v)||\nabla g||u|^{p-2}\right\|_{L^{\tilde{r}^\prime{(\mathbb{R}^2)}} }\lesssim\|\nabla  g\|_{L^2{(\mathbb{R}^2)}}\|u-v\|_{L^{\frac{4(2-\varepsilon)}{\varepsilon}}{(\mathbb{R}^2)}}\|u\|^{p-2}_{L^{\frac{4(2-\varepsilon)}{\varepsilon}(p-2)}{(\mathbb{R}^2)}} \\
    & \lesssim \|\nabla  g\|_{L^2} \|u-v\|_{H^1_\alpha{(\mathbb{R}^2)}} \|u\|^{p-2}_{H^1_\alpha{(\mathbb{R}^2)}}.
\end{aligned}
\end{equation}
Further,  \eqref{eq.es37} is 
modified as follows
\begin{equation}
    \begin{aligned}
     &   \left\||(u-v)| |\nabla \mathbb{G}_\omega|u|^{p-2}\right\|_{L^{\tilde{r}^\prime}{(\mathbb{R}^2)} }  \\
     &\lesssim    
 \|\nabla \mathbb{G}_\omega \|_{L^{2-\frac{\varepsilon}{2}}{(\mathbb{R}^2)}} \| |u-v| |u|^{p-2} \|_{L^{(2-\varepsilon)\frac{(4-\varepsilon)}{\varepsilon}}{(\mathbb{R}^2)}} \\
& \lesssim \| u-v \|_{L^{2(2-\varepsilon)\frac{(4-\varepsilon)}{\varepsilon}}{(\mathbb{R}^2)}} \| u \|^{p-2}_{L^{2(p-2)(2-\varepsilon)\frac{(4-\varepsilon)}{\varepsilon}}{(\mathbb{R}^2)}} \\
& \lesssim \|u-v\|_{H^1_\alpha{(\mathbb{R}^2)}} \|u\|^{p-2}_{H^1_\alpha{(\mathbb{R}^2)}} .
 \end{aligned}
\end{equation}

In a similar way, we can obtain also 
\begin{equation}\label{eq.u-v2}
    \left\| | (u-v)|\big(|u|^{p-1}+|v|^{p-1}\big)\right\|_{L^{\tilde{r}^\prime} {(\mathbb{R}^2)}} \lesssim \|u-v\|_{H^1_\alpha{(\mathbb{R}^2)}}(\| u\|^{p-1}_{ H^1_\alpha{(\mathbb{R}^2)}} + \left\| v\right\|^{p-1}_{ H^1_\alpha{(\mathbb{R}^2)}}),
\end{equation}
hence, we have \eqref{eq.con dif7}.

Considering $u$ and $v$ in $B_\mathcal{B}$, 
\begin{equation}\label{eq.fu-fv}
    \left\| f(u)-f(v)\right\|_{L^{\tilde{q}^\prime}(0,T) H^{1,\tilde{r}^\prime}_\alpha}\lesssim T^\frac{1}{\tilde{q}^\prime}\|u-v\|_{H^1_\alpha{(\mathbb{R}^2)}}
\end{equation}
follows from \eqref{eq.con dif7} and, choosing $T$ sufficiently small, \eqref{eq.ku-kv} and \eqref{eq.fu-fv} give \eqref{eq.contraction difference}.

\end{proof}

The same proof cannot be repeated for $d=3$ because in \eqref{eq.SH123}, i.e. 
\begin{equation}
    \left\|\int_0^t e^{i(t-\tau)\Delta_\alpha} u|u|^{p-1} d\tau\right\|_{L^q(0,T)H^{s,r}_\alpha} \lesssim  \left\| |u|^p \right\|_{L^{\tilde{q}^\prime}(0,T)H^{s,{\tilde{r}^\prime}}_\alpha  },
\end{equation}
 we have the condition $2\leq\tilde{r}<3$, that implies $\frac{3}{2}<\tilde{r}^\prime\leq2$, so we are out the range $\left(0,\frac{3}{2}\right)$, where we have the equality
 \begin{equation}
     H^{1,\tilde{r}^\prime}{(\mathbb{R}^3)}= H^{1,\tilde{r}^\prime}_\alpha{(\mathbb{R}^3)}
 \end{equation}
given in Theorem \ref{t.2.2}.

So we consider $s<1$ and prove a result similar to the local existence result obtained in \cite{K95}, but first we need some preliminary results.

We first recall the following estimate obtained in Lemma A.2 in \cite{K95} (see Chapter 2.5 in \cite{T00} and \cite{S95}).
\begin{Lemma} Assume  $p \in (1,2),$ $s \in (0,1)$ and  $1 < \ell, \ell_1< \infty,$ $1 < \ell_2  \leq \infty$ satisfy
\begin{equation}
\begin{aligned}
   &  \frac{1}{\ell} = \frac{1}{\ell_1} + \frac{p-1}{\ell_2}.
\end{aligned}
\end{equation}
Then we have
    \begin{equation}\label{eq.sk8}
        \begin{aligned}
            \||\phi|^{p-1}\phi\|_{H^{s,\ell}{(\mathbb{R}^3)}}\leq C \|\phi\|_{H^{s,\ell_1}{(\mathbb{R}^3)}}\|\phi\|^{p-1}_{L^{\ell_2}{(\mathbb{R}^3)}}
        \end{aligned}
    \end{equation}
    for all $\phi$ such that all norms in the right hand side are finite.
\end{Lemma}
This fractional chain rule can be generalized to the case where a singularity is present, adding some more {conditions}.
\begin{Proposition}\label{pcomp8}
    If  $p \in (1,2),$ $s \in (0,1)$ and  $3/2 < \ell \leq 2$, $\ell_1 \in [2,3),$  $ \ell_2 \in (3, \infty]$ satisfy
\begin{equation}\label{eq.el120}
\begin{aligned}
     & \frac{1}{\ell} + (p-1) \left(\frac{1}{\ell_1} - \frac{1}{\ell_2} \right) > \frac{(1+s)p}{3}  ,\\
   & \frac{1}{\ell} - \left(\frac{1}{\ell_1} - \frac{1}{\ell_2} \right) >  \frac{p}{3} ,\\
   &  \frac{1}{\ell_1} \geq \frac{s}3 .
\end{aligned}
\end{equation}
then for any smooth compactly supported functions $\varphi_1,\varphi_2$ we have
    \begin{equation}
        \begin{aligned}
            \left\|\left| \varphi_1 g+ \frac{c_0\varphi_2}{|x|}\right|^{p-1} \left(\varphi_1 g+ \frac{c_0\varphi_2}{|x|}\right)\right\|_{H^{s,\ell}(\mathbb{R}^3)}\leq C{(\varphi_1,\varphi_2)} \left(\|g\|_{H^{s,\ell_1}(\mathbb{R}^3)}+|c_0| \right) \left( \|g\|^{p-1}_{L^{\ell_2}(\mathbb{R}^3)} + |c_0| ^{p-1}\right)
        \end{aligned}
    \end{equation}
    for all $g \in H^{s,\ell_1}(\mathbb{R}^3) \cap L^{\ell_2}(\mathbb{R}^3) $ { and for any $c_0\in\mathbb{C}$.} 
\end{Proposition}

\begin{proof}
    We use the relation
    \begin{equation}\label{eq.hom0}
        \left| \varphi_1 g+ \frac{c_0\varphi_2}{|x|}\right|^{p-1} \left(\varphi_1 g+ \frac{c_0\varphi_2}{|x|}\right) = |x|^{-p\theta} \varphi_3
        \left| \widetilde{\varphi}_1 g+ \frac{c_0\varphi_2}{|x|^{1-\theta}}\right|^{p-1} \left(\widetilde{\varphi}_1 g+ \frac{c_0\varphi_2}{|x|^{1-\theta}}\right)
    \end{equation}
    where $\widetilde{\varphi}_1(x) = \varphi_1(x) |x|^{\theta}$ { and $\varphi_3$ is a smooth compactly supported function such that $\varphi_3(x)=1$ for every $x\in\operatorname{supp}\{\varphi_1\}\cup\operatorname{supp}\{\varphi_2\}.$}
    Next step is to check the estimate
    \begin{equation}\label{eq.hom1}
        \||x|^{-p\theta} \varphi_3 F\|_{H^{s,\ell}(\mathbb{R}^3)} \lesssim \| F\|_{H^{s,m}(\mathbb{R}^3)}
    \end{equation}
    provided
    \begin{equation}\label{eq.hom5}
        \frac{1}{m} < \frac{1}{\ell} - \frac{p\theta}{3}.
    \end{equation}
    Indeed using the fractional Leibniz rule, we can write
    \begin{equation}
    \begin{aligned}
        & \||x|^{-p\theta} \varphi_3 F\|_{H^{s,\ell}(\mathbb{R}^3)}  \lesssim 
          \||x|^{-p\theta} \varphi_3\|_{H^{s,\ell_3}(\mathbb{R}^3)}  \| F\|_{L^{\ell_4}(\mathbb{R}^3)}  \\
          & + \||x|^{-p\theta} \varphi_3\|_{L^{\ell_5}(\mathbb{R}^3)} \|F\|_{H^{s,m}(\mathbb{R}^3)}
    \end{aligned}
    \end{equation}
    with
    \begin{equation}
        \frac{1}{\ell} =  \frac{1}{\ell_3} + \frac{1}{\ell_4} = \frac{1}{\ell_5} + \frac{1}{m}
    \end{equation}
The boundedness of $$\||x|^{-p\theta} \varphi_3\|_{H^{s,\ell_3}(\mathbb{R}^3)} + \||x|^{-p\theta} \varphi_3\|_{L^{\ell_5}(\mathbb{R}^3)} $$
is guaranteed if
\begin{equation}
    p\theta + s < \frac{3}{\ell_3} , \ \ p\theta < \frac{3}{\ell_5}
\end{equation}

So we get
 \begin{equation}
    \begin{aligned}
        & \||x|^{-p\theta} \varphi_3 F\|_{H^{s,\ell}(\mathbb{R}^3)} \lesssim   \| F\|_{L^{\ell_4}(\mathbb{R}^3)} + \|F\|_{H^{s,m}(\mathbb{R}^3)}
    \end{aligned}
    \end{equation}
with
\begin{equation}\label{eq.ho3}
    \frac{1}{\ell_4} < \frac{1}{\ell} - \frac{p\theta + s}{3}, \ \  \frac{1}{m} < \frac{1}{\ell} - \frac{p\theta }{3}
\end{equation}
The Sobolev embedding leads to
\begin{equation}
    \begin{aligned}
        & \||x|^{-p\theta} \varphi_3 F\|_{H^{s,\ell}(\mathbb{R}^3)} \lesssim    \|F\|_{H^{s,m}(\mathbb{R}^3)}
    \end{aligned}
    \end{equation}
    provided
    \begin{equation}
        \frac{1}{m} - \frac{1}{\ell_4} = \frac{s}3
    \end{equation}
    so turning back to \eqref{eq.ho3}, we see that \eqref{eq.hom5} implies \eqref{eq.hom1}.

Now we can use the representation \eqref{eq.hom0} and quote \eqref{eq.hom1} in order to obtain
\begin{equation}
        \begin{aligned}
            \left\|\left| \varphi_1 g+ \frac{c_0\varphi_2}{|x|}\right|^{p-1} \left(\varphi_1 g+ \frac{c_0\varphi_2}{|x|}\right)\right\|_{H^{s,\ell}(\mathbb{R}^3)} \lesssim
            \left\|\left| \tilde{\varphi}_1 g+ \frac{c_0\varphi_2}{|x|^{1-\theta}}\right|^{p-1} \left(\tilde{\varphi}_1 g+ \frac{c_0\varphi_2}{|x|^{1-\theta}}\right)\right\|_{H^{s,m}(\mathbb{R}^3)}
                    \end{aligned}
        \end{equation}
Now we are in position to apply the {Staffilani-Kato} estimate \eqref{eq.sk8} with
$$ \phi = \tilde{\varphi}_1 g+ \frac{c_0\varphi_2}{|x|^{1-\theta}} $$
and obtain
\begin{equation}\label{eq.sk6}
        \begin{aligned}
            \||\phi|^{p-1}\phi\|_{H^{s,m}(\mathbb{R}^3)}\leq C \|\phi\|_{H^{s,\ell_1}(\mathbb{R}^3)}\|\phi\|^{p-1}_{L^{\ell_2}{(\mathbb{R}^3)}}
        \end{aligned}
    \end{equation}
with
\begin{equation}
    \frac{1}{m} = \frac{1}{\ell_1} + \frac{p-1}{\ell_2}.
\end{equation}
Note that 
\begin{equation}
     \|\phi\|_{H^{s,\ell_1}(\mathbb{R}^3)} \lesssim  \|\tilde{\varphi}_1 g\|_{H^{s,\ell_1}(\mathbb{R}^3)} +  |c_0|\||x|^{-1+\theta}\varphi_2\|_{H^{s,\ell_1}(\mathbb{R}^3)}
     \lesssim \| g\|_{H^{s,\ell_1}(\mathbb{R}^3)}  + |c_0|
\end{equation}
provided
\begin{equation}\label{ell1}
    (1-\theta)+s < \frac{3}{\ell_1}.
\end{equation}

Similarly,
\begin{equation}
     \|\phi\|_{L^{\ell_2}{(\mathbb{R}^3)}} \lesssim  \|\tilde{\varphi}_1 g\|_{L^{\ell_2}{(\mathbb{R}^3)}} +  |c_0|\||x|^{-1+\theta}\varphi_2\|_{L^{\ell_2}{(\mathbb{R}^3)}}
     \lesssim \| g\|_{H^{s,\ell_1}(\mathbb{R}^3)}  + |c_0|
\end{equation}
provided
\begin{equation} \label{ell2}
    (1-\theta) < \frac{3}{\ell_2}.
\end{equation}

The system inequalities 
\eqref{eq.hom5}, \eqref{ell1} and \eqref{ell2}  has the form
\begin{equation}
    \begin{aligned}
       &  \frac{1}{\ell_1} + \frac{p-1}{\ell_2} < \frac{1}{\ell} - \frac{p\theta}{3}, \\
       &   (1-\theta)+s < \frac{3}{\ell_1} \\
       & (1-\theta) < \frac{3}{\ell_2}\\
       & 0 \leq \theta < 1,
    \end{aligned}
\end{equation}
    so solving the system with respect to $\theta \in (0,1)$ we find
    \begin{equation}
    \begin{aligned}
       &  \frac{p\theta}{3} < \frac{1}{\ell} - \frac{1}{\ell_1} - \frac{p-1}{\ell_2}\\
       &   \theta > 1+s - \frac{3}{\ell_1}, \theta > 1- \frac{3}{\ell_2}, \\
       & 0 \leq \theta \leq 1,       
    \end{aligned}
\end{equation}
and the existence of $\theta$ is guaranteed if
\begin{equation}
\begin{aligned}
   &\frac{(1+s)p}{3} <  \frac{1}{\ell} + (p-1) \left(\frac{1}{\ell_1} - \frac{1}{\ell_2} \right) \\
   & \frac{p}{3} < \frac{1}{\ell} - \left(\frac{1}{\ell_1} - \frac{1}{\ell_2} \right)\\
   & s - \frac{3}{\ell_1} \leq  0, \ \frac{1}{\ell} - \frac{1}{\ell_1} - \frac{p-1}{\ell_2} \geq 0.
\end{aligned} 
\end{equation}
Note that first two requirements
\begin{equation}
\begin{aligned}
   &\frac{(1+s)p}{3} <  \frac{1}{\ell} + (p-1) \left(\frac{1}{\ell_1} - \frac{1}{\ell_2} \right) \\
   & \frac{p}{3} < \frac{1}{\ell} - \left(\frac{1}{\ell_1} - \frac{1}{\ell_2} \right)
\end{aligned} 
\end{equation}
imply
\begin{equation}
  \frac{1}{\ell} - \frac{1}{\ell_1} - \frac{p-1}{\ell_2}  =  \frac{1}{\ell} - \frac{1}{\ell_1} + \frac{1}{\ell_2}  - \frac{p}{\ell_2} > \frac{p}3 - \frac{p}{\ell_2} \geq 0 
\end{equation}
provided $\ell_2 \geq 3.$
\end{proof}

\begin{Corollary}\label{c.K1}
      If $s \in (0,1),$ $p \in (1,2)$ and $\ell \in (3/2,2]$ satisfy
      \begin{equation}\label{eq.ps3}
          p+s < \frac{3}{\ell},
      \end{equation}
      then for  any compactly supported functions $\varphi_1, \varphi_2$ there exists $C>0$ so that  for any $g \in H^{s,\ell}(\mathbb{R}^3)$ and any complex number $c_0$ we have 
 \begin{equation}\label{eq.esar77}
        \begin{aligned}
            \left\|\left| \varphi_1 g+ \frac{c_0\varphi_2}{|x|}\right|^{p-1} \left(\varphi_1 g+ \frac{c_0\varphi_2}{|x|}\right)\right\|_{H^{s,\ell}(\mathbb{R}^3)}\leq C \left(\|g\|_{H^{s,\ell_1}(\mathbb{R}^3)}+|c_0| \right)^p .
        \end{aligned}
    \end{equation}  
    for any $\ell_1 \in [2,3).$
\end{Corollary}

\begin{proof}
    We have to find $\ell_1\in [2,3)$, $\ell_2 \geq 3$ so that \eqref{eq.el120} is valid and then to apply 
Proposition \ref{pcomp8} in combination with Sobolev embedding $H^{s,\ell_1}{(\mathbb{R}^3)} \subset L^{\ell_2}{(\mathbb{R}^3)}$ that requires
$$ \frac{1}{\ell_1} - \frac{1}{\ell_2} = \frac{s}{3}.$$

 We shall take at the beginning $\ell_1=2$. Therefore  $$\frac{1}{\ell_2} = \frac{1}{2} - \frac{s}3$$
 and we have to verify
\begin{equation}\label{eq.el128}
\begin{aligned}
     & \frac{1}{\ell} + (p-1)\frac{s}{3}  > \frac{(1+s)p}{3}  ,\\
   & \frac{1}{\ell} - \frac{s}{3}  >  \frac{p}{3} ,\\
   &  \frac{1}{2} \geq \frac{s}3 .
\end{aligned}
\end{equation}
This system can be reduced to 
\eqref{eq.ps3}. Finally, to cover the case $\ell_1\in [2,3)$ we use the fact that multiplication by smooth compactly supported function is a bounded operator from $H^{s, \ell_1}{(\mathbb{R}^3)}$ to $H^{s, 2}{(\mathbb{R}^3)}$ for $\ell_1 \geq 2.$
    
\end{proof}

\begin{Corollary} \label{c.K2}
      If $s \in [0,1),$ $p \in (1,2)$  satisfy
      \begin{equation}\label{eq.ps9}
      \begin{aligned}
         &  p+s < 2,
      \end{aligned}
      \end{equation}
      then we have the following properties
      \begin{enumerate}
          \item[i)]  we can find a couple $\ell \in (3/2,2], \ell_1 \in [2,3)$, such that
      \begin{equation}\label{eq.exel8}
      \begin{aligned}
          &  \frac{p-2}{3}  <  \frac{p}{\ell_1} - \frac{1}{\ell} \leq \frac{s(p-1)}2 \ \mbox{if} \ \ p+s < \frac{3}2 , \\
          &   \frac{p-2}{3}  <  \frac{p}{\ell_1} - \frac{1}{\ell} \leq \frac{p}6 - \frac{s}3 ,
          \ \ \mbox{if} \ \  \frac{3}2 \leq p+s < 2, s < \frac{1}2, \\
          &   \frac{p-2}{3}  <  \frac{p}{\ell_1} - \frac{1}{\ell} \leq \frac{s(p-1)}3, \ \mbox{if} \ \  \frac{3}2 \leq p+s < 2, s \geq \frac{1}2,
      \end{aligned}
    \end{equation} 
    \item[ii)] for any couple  $\ell \in (3/2,2], \ell_1 \in [2,3),$ satisfying \eqref{eq.exel8} we have $s < 3/\ell_1$
      and   we have the estimate 
 \begin{equation}\label{eq.ssce6}
        \begin{aligned}
            \left\| \phi|\phi|^{p-1}\right\|_{H^{s,\ell}(\mathbb{R}^3)}\leq C \|\phi\|^p_{H^{s,\ell_1}_\alpha(\mathbb{R}^3)}
        \end{aligned}
    \end{equation}  
    provided $\phi \in H^{s,\ell_1}_\alpha(\mathbb{R}^3).$ 
      \end{enumerate}    
\end{Corollary}
\begin{proof}
    For any $\phi \in H^{s,\ell_1}_\alpha(\mathbb{R}^3)$ with $s < 3/\ell_1$,
    we have
    \begin{equation}\label{eq.phi}
    \begin{aligned}
        \phi |\phi|^{p-1} &= {\varphi\phi|\phi|^{p-1}+(1-\varphi)\phi|\phi|^{p-1}}\\
        &=\varphi  (\varphi_1\phi) |\varphi_1\phi|^{p-1} + (1-\varphi)  (1-\varphi_2)\phi |(1-\varphi_2)\phi|^{p-1}
    \end{aligned}
    \end{equation}
    where $\varphi, \varphi_1, \varphi_2$ are smooth compactly supported functions that are $1$ near the origin
    and 
    \begin{equation}
        \begin{aligned}
            & \varphi(x) = 1 \ \ \mbox{for $x \in \mathrm{supp} \ \varphi_2$} \\
            & \varphi_1(x) = 1 \ \ \mbox{for $x \in \mathrm{supp} \ \varphi$}.
        \end{aligned}
    \end{equation}
    { In particular  $1-\varphi_2(x) = 1$ for $x \in \operatorname{supp} (1-\varphi)$.}
    We call $\phi_1 = \varphi_1\phi$, that is a function with localised singularity near zero, described in Proposition \ref{pcomp8},
while $\phi_2= (1-\varphi_2)\phi \in H^{s,\ell_1}(\mathbb{R}^3)$ is zero near the origin.

Applying the estimate \eqref{eq.esar77} we can write
\begin{equation}
        \begin{aligned}
            \left\|\varphi \phi_1|\phi_1|^{p-1}\right\|_{H^{s,\ell}(\mathbb{R}^3)}\lesssim \|\phi_1\|^p_{H^{s,\ell_1}_\alpha(\mathbb{R}^3)}
            \lesssim \|\phi\|^p_{H^{s,\ell_1}_\alpha(\mathbb{R}^3)}
        \end{aligned}
    \end{equation}  
provided
 \begin{equation}\label{eq.ps51}
          p+s < \frac{3}{\ell}.
      \end{equation}
Further we apply the estimate \eqref{eq.sk8} for $\phi_2$, combined with the Sobolev embedding, so we have
 \begin{equation}\label{eq.sk55}
        \begin{aligned}
            \|(1-\varphi) \phi_2 |\phi_2|^{p-1}\|_{H^{s,\ell}{(\mathbb{R}^3)}}\leq C \|\phi_2\|^p_{H^{s,\ell_1}{(\mathbb{R}^3)}} \lesssim \|\phi\|^p_{H^{s,\ell_1}_\alpha{(\mathbb{R}^3)}}
        \end{aligned}
    \end{equation}
provided
\begin{equation}\label{eq.ps51 2}
\begin{aligned}
   &  \frac{1}{\ell} \geq \frac{1}{\ell_1} + (p-1) \left( \frac{1}{\ell_1}- \frac{s}{3}\right) =  \frac{p}{\ell_1}- \frac{s(p-1)}{3}.
\end{aligned}
\end{equation}
{The condition \eqref{eq.ps51 2} ensures that the Sobolev embedding $H^{s,\ell_1}(\mathbb{R}^3)\subseteq L^{\ell_2}(\mathbb{R}^3)$, with $\ell_2$ defined in \eqref{eq.sk8}.}
Therefore we have \eqref{eq.ssce6} provided \eqref{eq.ps51}, {\eqref{eq.ps51 2}} and we can find $\ell \in (3/2,2],$ $\ell_1 \in [2,3)$ so that
\begin{equation}
\begin{aligned}
   & \frac{s(p-1)}{3} \geq \frac{p}{\ell_1} -  \frac{1}{\ell},\\
    & \frac{ p+s}{3} < \frac{1}{\ell}.
\end{aligned}
\end{equation}
To study the existence of the couple $(\ell, \ell_1)$ we consider two cases.
In the case
\begin{equation}
    p+s < \frac{3}2
\end{equation}
the solution exists if and only if  
      \begin{equation}
     \begin{aligned}
     \frac{p-2}{3}  <  \frac{s(p-1)}3
     \end{aligned}
   \end{equation}
that is always fulfilled for $p \in (1,2).$
   If 
    $3/2 \leq p+s < 2,$ then we need
    \begin{equation}
          \frac{p-2}{3}  <   \min \left(\frac{p}6 - \frac{s}3 ,  \frac{s(p-1)}3 \right).
    \end{equation}
    and this is always fulfilled, when $p>1, s \in (0,1)$ and $p+s < 2.$

    If $p+s>2,$ then there is no  solution.

\end{proof}
The inequality \eqref{eq.ssce6} was the last tool needed for treating the case $d=3$.
\begin{proof}[Proof of Theorem \ref{t.Schrodinger3}]
      Consider again the operator
      \begin{equation}\label{eq.OPK7}
          \mathfrak{K} (u) = e^{it \Delta_\alpha} u_0 -i {\mu}\int_0^t e^{i(t-\tau)\Delta_\alpha}  u(\tau)|u(\tau)|^{p-1} d\tau, 
      \end{equation}

  First we  apply  Corollary \ref{c.K2} and find $\ell \in (3/2,2], \ell_1 \in [2,3)$ so that
  \begin{equation}\label{eq.cin8}
      p+s < \frac{3}\ell
  \end{equation}
  and  
  \eqref{eq.ssce6} holds.
  
Further we follow the proof of Theorem 4.1 in \cite{K95} and Section 4.4 in \cite{Cazenave} and introduce the space
\begin{equation}
    \mathcal{Y}_s = \bigcap_{\mbox{$q,r$ satisfy \eqref{eq.adm97}}} L^q{(0,T)} H^{s,r}_\alpha ,
\end{equation}
where
\begin{equation}\label{eq.adm97}
    \frac{2}{q} + \frac{3}{r} = \frac{3}{2}, \mbox{ with  $r\in[2,3),\ q\in (4,\infty]$}
\end{equation}

Then Strichartz estimates imply
\begin{equation}
  \|e^{it \Delta_\alpha} u_0\|_{\mathcal{Y}_s} \lesssim \|u_0\|_{H^s_\alpha{(\mathbb{R}^3)}} \leq  R. 
\end{equation}

The application of Strichartz estimate for the Duhamel integral in \eqref{eq.OPK7} gives
\begin{equation}
    \left\| \int_0^t e^{i(t-\tau)\Delta_\alpha}  u(\tau)|u(\tau)|^{p-1} d\tau\right\|_{\mathcal{Y}_s} \lesssim
    \|u|u|^{p-1} \|_{L^{m}{(0,T)} H^{s,\ell}_\alpha}
\end{equation}
  such that
  \begin{equation}
      \frac{7}2 = \frac{2}m + \frac{3}\ell, \ \ m \in (1,4/3).
  \end{equation}
The inequality \eqref{eq.cin8}  and Theorem \ref{t.2.2} imply
\begin{equation}
    \|u|u|^{p-1} \|_{L^{m}{(0,T)} H^{s,\ell}_\alpha}= \|u|u|^{p-1} \|_{L^{m}{(0,T)} H^{s,\ell}}
\end{equation}
so we can apply the estimate \eqref{eq.ssce6} and deduce
\begin{equation}
     \|u|u|^{p-1} \|_{ H^{s,\ell}{(\mathbb{R}^3)}} \lesssim  \|u \|^p_{ H^{s,\ell_1}_\alpha{(\mathbb{R}^3)}}
\end{equation}
Then we can write
\begin{equation}
     \|u|u|^{p-1} \|_{L^m{(0,T)} H^{s,\ell}} \lesssim  \|u \|^p_{ L^{pm}{(0,T)} H^{s,\ell_1}_\alpha} \lesssim T^{1/m-p/m_1} 
     \|u \|^p_{ L^{m_1}{(0,T)} H^{s,\ell_1}_\alpha}
\end{equation}
where We have chosen admissible Strichartz couple $(m_1,\ell_1)$ so that
\begin{equation}
    \frac{p}{m_1} < \frac{1}m, \ \frac{2}{m_1} + \frac{3}{\ell_1}= \frac{3}2.
\end{equation}

In conclusion we arrive at
\begin{equation}
    \| \mathfrak{K} (u) \|_{\mathcal{Y}_s} \lesssim R +  T^{1/m-p/m_1} \|u\|_{\mathcal{Y}_s}^p.
\end{equation}
so taking $T>0$ sufficiently small we can apply the contraction argument of Kato \cite{K95} we complete the proof.

\end{proof}
\section{Data availability} No data were created or analyzed in this article

\appendix

\section{Integral  estimate}

We shall use the following
\begin{Lemma}\label{l.ste44}
    Let $s {<2}$ and $A$ { real constant}.  Let $f({x})$ be a positive decreasing measurable function on $(0,\infty),$ such that
\begin{equation}
    \int_0^\infty \sigma^{-s/2+A-1} f(\sqrt{\sigma}) d\sigma < \infty.
\end{equation} 
Then for any $\omega>0$ there exists a constant $C=C(s,A)$ so that we have
    \begin{equation}\label{eq.inin3}
        \int_0^\infty t^{-s/2} (\omega+t)^{A-1} f(\sqrt{\omega+t}r) dt \leq C  r^{s-2A}, \ \forall r>0.
    \end{equation}
\end{Lemma}
\begin{proof}
   We use the assumptions of the Lemma and we can write 
    \begin{equation}
    \begin{aligned}
       \int_0^\infty t^{-s/2} (\omega+t)^{A-1} f(\sqrt{\omega+t}r) dt& \lesssim 
       \int_0^\infty t^{-s/2+A-1}  f(\sqrt{t}r) dt  \\
       & = r^{s-2A} \int_0^\infty \sigma^{-s/2+A-1} f(\sqrt{\sigma}) d\sigma.
    \end{aligned}
    \end{equation}
\end{proof}

\bibliographystyle{plain}
 \bibliography{SING}

\begin{thebibliography}{10}

\bibitem{ABCT22}
Riccardo Adami, Filippo Boni, Raffaele Carlone, and Lorenzo Tentarelli.
\newblock {Existence, structure, and robustness of ground states of a NLSE in 3D with a point defect}.
\newblock {\em Journal of Mathematical Physics}, 63(7):071501, 07 2022.

\bibitem{ABCT}
Riccardo Adami, Filippo Boni, Raffaele Carlone, and Lorenzo Tentarelli.
\newblock Ground states for the planar {NLSE} with a point defect as minimizers of the constrained energy.
\newblock {\em Calc. Var. Partial Differential Equations}, 61(5):Paper No. 195, 32, 2022.

\bibitem{AH81}
Sergio Albeverio and Raphael H\o{}egh-Krohn.
\newblock Point interactions as limits of short range interactions.
\newblock {\em J. Operator Theory}, 6(2):313--339, 1981.

\bibitem{AS61}
Nachman Aronszajn and K.~T. Smith.
\newblock Theory of {Bessel} potentials. {I}.
\newblock {\em Annales de l'Institut Fourier}, 11:385--475, 1961.

\bibitem{BGR25}
Daniele Barbera, Vladimir Georgiev, and Mario Rastrelli.
\newblock On the cauchy problem for the reaction-diffusion system with point-interaction in {$\mathbb R^2$}.
\newblock {\em arXiv}, 2504.08460, 2025.

\bibitem{BF61}
F.~A. Berezin and L.~D. Faddeev.
\newblock Remark on the {S}chr\"{o}dinger equation with singular potential.
\newblock {\em Dokl. Akad. Nauk SSSR}, 137:1011--1014, 1961.

\bibitem{BCS2015}
Jorge~J. Betancor, Alejandro~J. Castro, and Pablo~Ra\'{u}l Stinga.
\newblock The fractional {B}essel equation in {H}\"{o}lder spaces.
\newblock {\em J. Approx. Theory}, 184:55--99, 2014.

\bibitem{BS87}
M.~S. Birman and M.~Z. Solomjak.
\newblock {\em Spectral Theory of Self-Adjoint Operators in Hilbert Space}, volume~5 of {\em Mathematics and its Applications}.
\newblock D. Reidel Publishing Company, Dordrecht, Holland, Dordrecht, 1 edition, 1987.
\newblock Originally published in Russian. Part of the Springer Book Archive. Copyright 1987.

\bibitem{CFN21}
Claudio Cacciapuoti, Domenico Finco, and Diego Noja.
\newblock Well posedness of the nonlinear schrödinger equation with isolated singularities.
\newblock {\em Journal of Differential Equations}, 305:288--318, 2021.

\bibitem{Cazenave}
Thierry Cazenave.
\newblock {\em Semilinear {S}chr\"{o}dinger equations}, volume~10 of {\em Courant Lecture Notes in Mathematics}.
\newblock New York University, Courant Institute of Mathematical Sciences, New York; American Mathematical Society, Providence, RI, 2003.

\bibitem{CMY19}
Horia~D. Cornean, Alessandro Michelangeli, and Kenji Yajima.
\newblock Two-dimensional {S}chr\"{o}dinger operators with point interactions: threshold expansions, zero modes and {$L^p$}-boundedness of wave operators.
\newblock {\em Rev. Math. Phys.}, 31(4):1950012, 32, 2019.

\bibitem{CMY19b}
Horia~D. Cornean, Alessandro Michelangeli, and Kenji Yajima.
\newblock Erratum: {T}wo-dimensional {S}chr\"{o}dinger operators with point interactions: threshold expansions, zero modes and {$L^p$}-boundedness of wave operators.
\newblock {\em Rev. Math. Phys.}, 32(4):2092001, 5, 2020.

\bibitem{DMSY18}
Gianfausto Dell'Antonio, Alessandro Michelangeli, Raffaele Scandone, and Kenji Yajima.
\newblock {$L^p$}-boundedness of wave operators for the three-dimensional multi-centre point interaction.
\newblock {\em Ann. Henri Poincar\'{e}}, 19(1):283--322, 2018.

\bibitem{EN2000}
Klaus-Jochen Engel and Rainer Nagel.
\newblock {\em One-parameter semigroups for linear evolution equations}, volume 194 of {\em Graduate Texts in Mathematics}.
\newblock Springer-Verlag, New York, 2000.
\newblock With contributions by S. Brendle, M. Campiti, T. Hahn, G. Metafune, G. Nickel, D. Pallara, C. Perazzoli, A. Rhandi, S. Romanelli and R. Schnaubelt.

\bibitem{FN23}
Domenico Finco and Diego Noja.
\newblock Blow-up and instability of standing waves for the nls with a point interaction in dimension two.
\newblock {\em Zeitschrift für angewandte Mathematik und Physik}, 74(4), 2023.

\bibitem{FG24}
Luigi Forcella and Vladimir Georgiev.
\newblock Local well-posedness and blow-up in the energy space for the 2d nls with point interaction.
\newblock {\em arXiv}, 2410.16039, 2024.

\bibitem{FGI22}
Noriyoshi Fukaya, Vladimir Georgiev, and Masahiro Ikeda.
\newblock On stability and instability of standing waves for 2d-nonlinear schr\"odinger equations with point interaction.
\newblock {\em Journal of Differential Equations}, 321:258--295, 2022.

\bibitem{GMS18}
Vladimir Georgiev, Alessandro Michelangeli, and Raffaele Scandone.
\newblock On fractional powers of singular perturbations of the {L}aplacian.
\newblock {\em J. Funct. Anal.}, 275(6):1551--1602, 2018.

\bibitem{GMS24}
Vladimir Georgiev, Alessandro Michelangeli, and Raffaele Scandone.
\newblock Standing waves and global well-posedness for the 2d hartree equation with a point interaction.
\newblock {\em Communications in Partial Differential Equations}, 49(3):242--278, 2024.

\bibitem{GR25}
Vladimir Georgiev and Mario Rastrelli.
\newblock {S}obolev spaces for singular perturbation of {2D} {L}aplace operator.
\newblock {\em Nonlinear Analysis}, 251:113710, 2025.

\bibitem{H81}
Daniel Henry.
\newblock {\em Geometric theory of semilinear parabolic equations}, volume 840 of {\em Lecture Notes in Mathematics}.
\newblock Springer-Verlag, Berlin-New York, 1981.

\bibitem{K95}
Tosio Kato.
\newblock On nonlinear {S}chr\"{o}dinger equations. {II}. {$H^s$}-solutions and unconditional well-posedness.
\newblock {\em J. Anal. Math.}, 67:281--306, 1995.

\bibitem{Ko66}
Hikosaburo Komatsu.
\newblock Fractional powers of operators.
\newblock {\em Pacific J. Math.}, 19:285--346, 1966.

\bibitem{Ko67}
Hikosaburo Komatsu.
\newblock Fractional powers of operators. {II}. {I}nterpolation spaces.
\newblock {\em Pacific J. Math.}, 21:89--111, 1967.

\bibitem{MOS18}
Alessandro Michelangeli, Alessandro Olgiati, and Raffaele Scandone.
\newblock Singular {H}artree equation in fractional perturbed {S}obolev spaces.
\newblock {\em J. Nonlinear Math. Phys.}, 25(4):558--588, 2018.

\bibitem{OSY12}
Noboru Okazawa, Toshiyuki Suzuki, and Tomomi Yokota.
\newblock Energy methods for abstract nonlinear {S}chr\"{o}dinger equations.
\newblock {\em Evol. Equ. Control Theory}, 1(2):337--354, 2012.

\bibitem{Simon73b}
B.~Simon.
\newblock Essential self-adjointness of schrödinger operators with singular potentials.
\newblock {\em Arch. Rational Mech. Anal.}, 52:44--48, 1973.

\bibitem{S95}
Gigliola Staffilani.
\newblock {\em The initial value problem for some dispersive differential equations}.
\newblock ProQuest LLC, Ann Arbor, MI, 1995.
\newblock Thesis (Ph.D.)--The University of Chicago.

\bibitem{T00}
Michael~E. Taylor.
\newblock {\em Tools for {PDE}}, volume~81 of {\em Mathematical Surveys and Monographs}.
\newblock American Mathematical Society, Providence, RI, 2000.
\newblock Pseudodifferential operators, paradifferential operators, and layer potentials.

\end{thebibliography}
\end{document}